\documentclass{article}

\usepackage{arxiv}

\usepackage[utf8]{inputenc} % allow utf-8 input
\usepackage[T1]{fontenc}    % use 8-bit T1 fonts
\usepackage{hyperref}       % hyperlinks
\usepackage{url}            % simple URL typesetting
\usepackage{booktabs}       % professional-quality tables
\usepackage{amsfonts}       % blackboard math symbols
\usepackage{nicefrac}       % compact symbols for 1/2, etc.
\usepackage{microtype}      % microtypography
\usepackage{lipsum}
\usepackage{amsthm} % Can be removed after putting your text content
\usepackage{graphicx}
\usepackage{doi}

\title{Grouped Transformations and Regularization in High-Dimensional Explainable ANOVA Approximation}

\author{ 	Felix Bartel \\
	Faculty of Mathematics\\
	Chemnitz University of Technology\\
	09107 Chemnitz \\
	\texttt{felix.bartel@math.tu-chemnitz.de} \\ 
	\And
	\href{https://orcid.org/0000-0003-3651-4364}{\includegraphics[scale=0.06]{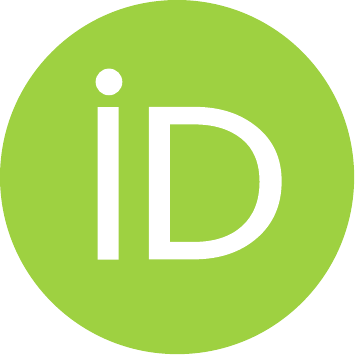}\hspace{1mm}Daniel Potts} \\
	Faculty of Mathematics\\
	Chemnitz University of Technology\\
	09107 Chemnitz \\
	\texttt{potts@math.tu-chemnitz.de} \\
	\And
	\href{https://orcid.org/0000-0003-1152-0864}{\includegraphics[scale=0.06]{orcid.pdf}\hspace{1mm}Michael Schmischke} \\
	Faculty of Mathematics\\
	Chemnitz University of Technology\\
	09107 Chemnitz \\
	\texttt{michael.schmischke@math.tu-chemnitz.de} \\
}

\renewcommand{\shorttitle}{Grouped Transformations in High-Dimensional Explainable ANOVA Approximation}

\hypersetup{
pdftitle={Grouped Transformations and Regularization in High-Dimensional Explainable ANOVA Approximation},
pdfsubject={scientific computing},
pdfauthor={Felix Bartel, Daniel Potts, Michael Schmischke},
pdfkeywords={analysis of variance, ANOVA, explainable approximation, fast iterative shrinkage-thresholding algorithm, FISTA, group lasso, high dimensional approximation, multivariate trigonometric polynomials, nonequispaced fast Fourier transform, NFFT, LSQR},
}
\usepackage{amsfonts, mathtools, amssymb, amsopn, bm, mathtools}
\usepackage{mathrsfs, cleveref}
\usepackage{algorithm, algorithmic}
\usepackage{faktor}
\usepackage{booktabs, siunitx}
\usepackage{multirow}
\usepackage[dvipsnames]{xcolor}
\usepackage{tikz, pgfplots, subcaption, graphicx}
\pgfplotsset{compat=1.13}
\usepackage{catchfile}
\usepackage{float}
\usepackage{enumerate}

\usepackage{todonotes}
\usepackage{epstopdf}
\DeclareMathOperator{\diag}{diag}

\DeclareMathOperator*{\supp}{supp}

\newcommand{\fun}[3]{#1 \colon #2 \rightarrow #3}
\newcommand{\abs}[1]{\left|#1\right|}
\newcommand{\norm}[2]{\left\| #1 \right\|_{#2}}

\newcommand{\herm}{^\mathsf H}

\renewcommand{\subset}{\subseteq}
\renewcommand{\epsilon}{\varepsilon}
\renewcommand{\b}{\bm}
\renewcommand{\cref}{\Cref}

\newcommand{\eps}{\mathrm\varepsilon}
\newcommand{\e}{\mathrm e}
\renewcommand{\i}{\mathrm i}

\newcommand{\C}{\mathbb C}
\newcommand{\N}{\mathbb N}
\newcommand{\R}{\mathbb R}
\newcommand{\Z}{\mathbb Z}
\newcommand{\T}{\mathbb T}

\newcommand{\x}{\b x}
\newcommand{\y}{\b y}
\renewcommand{\k}{\b k}

\newcommand{\F}{\b F}

\renewcommand{\L}{\mathrm{L}}

\newcommand{\sobolev}[1]{\mathrm{H}^{#1}}

\newcommand{\D}{\mathcal D}
\renewcommand{\u}{\b u}
\newcommand{\uc}{\b{u}^{\mathrm c}}
\newcommand{\au}{\abs{\u}}
\newcommand{\auc}{d-\abs{\u}}
\renewcommand{\v}{\bm v}

\newcommand{\va}[1]{\sigma^2(#1)} 
\newcommand{\gsi}[2]{\varrho(#1,#2)}

\newcommand{\Fseries}[4]{\sum_{#1 \in #2} #3 \, \e^{2\pi\i #1 \cdot #4}}

\newcommand{\fc}[2]{\mathrm{c}_{#1}\!\left(#2\right)}

\usepackage{algorithmic}

\newcommand{\drawfrequencyset}[1]{
	\begin{tikzpicture}[scale = 0.25]
		\CatchFileDef\loadeddata{#1}{\endlinechar=-1}
		\foreach \x\y\z\color in \loadeddata{
			\draw[fill=black!25!\color] (\x-0.5,\y-0.5,\z+0.5) -- (\x+0.5,\y-0.5,\z+0.5) -- (\x+0.5,\y+0.5,\z+0.5) -- (\x-0.5,\y+0.5,\z+0.5) -- cycle;
			\draw[fill=black!50!\color] (\x+0.5,\y-0.5,\z+0.5) -- (\x+0.5,\y-0.5,\z-0.5) -- (\x+0.5,\y+0.5,\z-0.5) -- (\x+0.5,\y+0.5,\z+0.5) -- cycle;
			\draw[fill=\color] (\x+0.5,\y+0.5,\z+0.5) -- (\x+0.5,\y+0.5,\z-0.5) -- (\x-0.5,\y+0.5,\z-0.5) -- (\x-0.5,\y+0.5,\z+0.5) -- cycle;
		}
	\end{tikzpicture}
}

\newtheorem{theorem}{Theorem}[section]
\newtheorem{lemma}[theorem]{Lemma}
\newtheorem{remark}[theorem]{Remark}
\newtheorem{definition}[theorem]{Definition}
\newtheorem{example}[theorem]{Example}
\newtheorem{corollary}[theorem]{Corollary}

\begin{document}
\maketitle

\begin{abstract}
	In this paper we propose a tool for high-dimensional approximation based on trigonometric polynomials where we allow only low-dimensional interactions of variables. In a general high-dimensional setting, it is already possible to deal with special sampling sets such as sparse grids or rank-1 lattices. This requires black-box access to the function, i.e., the ability to evaluate it at any point. Here, we focus on scattered data points and grouped frequency index sets along the dimensions. From there we propose a fast matrix-vector multiplication, the grouped Fourier transform, for high-dimensional grouped index sets. Those transformations can be used in the application of the previously introduced method of approximating functions with low superposition dimension based on the analysis of variance (ANOVA) decomposition where there is a one-to-one correspondence from the ANOVA terms to our proposed groups. The method is able to dynamically detected important sets of ANOVA terms in the approximation. In this paper, we consider the involved least-squares problem and add different forms of regularization: Classical Tikhonov-regularization, namely, regularized least squares and the technique of group lasso, which promotes sparsity in the groups. As for the latter, there are no explicit solution formulas which is why we applied the fast iterative shrinking-thresholding algorithm to obtain the minimizer. Moreover, we discuss the possibility of incorporating smoothness information into the least-squares problem. Numerical experiments in under-, overdetermined, and noisy settings indicate the applicability of our algorithms. While we consider periodic functions, the idea can be directly generalized to non-periodic functions as well.
\end{abstract}

% keywords can be removed

\section{Introduction}

Discrete transformations like the discrete Fourier transform, the discrete cosine transform, and many others play an important role in a large variety of applications in applied mathematics and other sciences.
They have since given rise to algorithms like the fast Fourier transform (FFT) and the fast cosine transform that allow for the fast multiplication of the associated matrices, cf.~\cite{PlPoStTa18}.
The FFT with its many applications belongs to the most important algorithms of our time.
Those concepts have been generalized to allow for nodes to be nonequispaced resulting in the nonequispaced fast Fourier transform (NFFT) and the nonequispaced fast cosine transform, cf.~\cite{KeKuPo09, PlPoStTa18, nfft3}.
While those algorithms have a good complexity, it depends on the size of the frequency index set which grows exponentially in the dimension for many common examples.

For high dimensional approximation with trigonometric polynomials there already exist efficient methods for special sampling sets, such as sparse grids \cite{He01, JiXu10, GriHa13} or rank-1 lattice nodes \cite{Kae2012,KaPoVo13}.
In \cite{PoSc19a} the multivariate analysis of variance (ANOVA) decomposition, cf.~\cite{CaMoOw97, RaAl99, LiOw06, KuSlWaWo09, Holtz11, Gu2013}, was related to a decomposition in the frequency domain that allowed for the Fourier coefficients to be considered in subsets of coordinate indices related to the ANOVA terms which we propose as groups.
If one assumes sparsity in the index sets related to sparsity in the ANOVA decomposition, it is possible to decompose high-dimensional nonequispaced discrete transforms into multiple transforms associated with ANOVA terms.
In this paper, we introduce the grouped Fourier transformation which is able to handle high-dimensional frequency index sets if they have sparsity in their support, i.e., the number of nonzero elements in the frequencies is limited by a so called superposition threshold, cf.~\cite{PoSc19a}.

In \cite{PoSc19a} a method was proposed for the approximation of periodic functions with a low superposition dimension, see e.g.~\cite{CaMoOw97, KuSlWaWo09, DePeVo10, Ow19, HaSchaeShiToTrWa21}, and a sparse ANOVA decomposition. In other words, we assume that the variables of a function only interact in low-dimensional terms. From a theoretical viewpoint, it has been proven in \cite[Section 4]{PoSc19a} that periodic functions of specific dominating-mixed smoothness with product and order-dependent weights yield low superposition dimensions and can therefore be approximated well by the method. In applications, one may assume that we have sparsity-of-effects or that the Pareto principle holds, see e.g.~\cite{Wu2011,HaSchaeShiToTrWa21}, meaning that many real world systems are dominated by only a small number of low complexity interactions. The frequency index sets used in the method have a grouped structure and we have a one-to-one relation to the truncated ANOVA decomposition. The central part of the approximation approach is given by a least-squares problem with a Fourier matrix of a grouped index set which can be handled by the proposed grouped Fourier transform. Moreover, sensitivity analysis, cf.~\cite{OkHa04}, is used in order to rank the influence of the dimension interactions or ANOVA terms and dynamically detect important ones. This can be displayed or interpreted as an explainable ANOVA network and may lead to a further truncation of the ANOVA decomposition to an \emph{active set} of terms. One main advantage of the Fourier approach is that we are able to immediately compute an approximation to the variance of the ANOVA terms. The idea is in principle related to the multivariate decomposition method, cf.~\cite{KuNuPlSlWa17, GiKuNuWa18}, for the approximation of high- to infinite-dimensional integrals as well as functions, see \cite{WaWo10}. However, we are considering the approximation of functions and work with random or scattered data instead of quasi-Monte Carlo methods. Moreover, our detection of the active set is dynamic and comes from sensitivity analysis instead of a-priori information.

In this paper, we use the grouped Fourier transform to solve the appearing least-squares problem and discuss enhancements by combing it with regularization approaches: The first method uses least-squares minimization with the LSQR algorithm, see~\cite{PaSa82}, as in \cite{PoSc19a}. However, we add a Tikhonov regularization term and also discuss the incorporation of smoothness where we use Sobolev weights to adapt our algorithm for different decay properties of the Fourier coefficients. Here, the sensitivity analysis on the approximation lets us identify important terms which lead to an active set. Solving the least-squares system with this active set yields the benefit of a simpler model function and improves approximation quality, cf.\ results in \cite{PoSc19a}.

In the second approach we try to combine both steps, i.e., approximation and identification of an active set of terms by using the group lasso approach from \cite{YL06} which immediately promotes sparsity in the groups. This approach has been used and adapted to several scenarios, see e.g.\ \cite{MeGeBu2008, YangXKL10, SiFrHaTi2013}. Since we set the groups as the ANOVA terms, it refers to sparsity in the ANOVA decomposition, i.e., the identification of an active set of terms. This leads to a non-linear minimization problem, which we solve by using the fast iterative shrinking-thresholding Algorithm (FISTA), cf.~\cite{BT09}.
To our knowledge this regularization was only done for $\ell_2$-norms of groups, but here we were able to incorporate smoothness information by using Sobolev-type norms.

The paper is organized as follows: In \cref{sec:anova} we reiterate on the ANOVA decomposition for periodic functions as discussed in \cite{PoSc19a} and introduce grouped index sets. In \cref{sec:gft} we introduce the grouped Fourier transform for frequency sets of a grouped structure and propose a fast algorithm to compute it. Furthermore, we discuss complexity as well as the error of this algorithm. We apply the grouped Fourier transform in \cref{sec:approx} to the least-squares problem \eqref{min_prob} of approximating high-dimensional functions with the method proposed in \cite{PoSc19a}. In the \cref{sec:lsqr} we add a $\ell_2$ regularization term and therein use information about the smoothness of the function. The solver in this case is the well-known LSQR method, cf.~\cite{PaSa82}. \cref{sec:fista} contains the group lasso regularization idea using FISTA, see \cref{algo:fista}, as a solver. Here, the idea to use smoothness information is considered as well. We provide numerical examples with a special test function for both methods. In addition, we show a possible extension to non-periodic functions with an example from applications in \cref{sec:adult}.

\section{Classical ANOVA Decomposition}
\label{sec:anova}

In this section we discuss the main properties of the \emph{classical ANOVA decomposition}, see \cite{CaMoOw97, RaAl99, LiOw06, KuSlWaWo09, Holtz11, Gu2013}, for periodic functions $\fun{f}{\T^d}{\C}$ defined over the $d$-dimensional torus $\T \coloneqq \faktor{\R}{\Z}$. The decomposition has proven useful in understanding the reason behind the success of certain quadrature methods for high-dimensional integration \cite{Ni92, BuGr04,GrHo10} and also infinite-dimensional integration \cite{BaGn14, GrKuSl16, KuNuPlSlWa17}. For a general description of multivariate decompositions, we refer to \cite{KuSlWaWo09} and for the classical ANOVA decomposition of periodic functions to \cite{PoSc19a}.

For a function $f \in \L_2(\T^d)$ with spatial dimension $d \in \N$, we introduce the integral projections \begin{equation}\label{scrubs}
	\mathrm{P}_{\u} f(\x) = \int_{\T^{\auc}} f(\x) \mathrm{d} \x_{\uc}
\end{equation} for a subset of coordinate indices $\u \subset \D \coloneqq\{1,2,\dots,d\}$ and its complement $\uc = \D \setminus \u$. Additionally, for vectors $\x \in \C^d$ indexed with a subset $\u \subset \D$ we define $\x_{\u} \coloneqq (x_i)_{i \in \u}$. This leads to the ANOVA terms \begin{equation}\label{eq:anova:term_expansion}
	f_{\u}(\x) \coloneqq \mathrm{P}_{\u} f(\x) - \sum_{\v\subsetneq\u} f_{\v}(\x) = \sum_{\substack{\k \in \Z^d \\ \supp\k = \u}} \fc{\k}{f} \e^{2\pi\i\langle\k,\x\rangle}, \u \subset \D,
\end{equation} with $\supp \k = \{ j \in \D \colon k_j \neq 0 \}$, and Fourier coefficients \begin{equation*}
	\fc{\k}{f} = \int_{\mathbb T^d} f(\x) \,\e^{- 2\pi\i\langle\k,\x\rangle} \,\mathrm{d} \x.
\end{equation*} Then the classical ANOVA decomposition of $f$ is given by \begin{equation*}
	f = \sum_{\u \subset \D} f_{\u}.
\end{equation*} The uniqueness of the decomposition as well as the relation to the series expansion \eqref{eq:anova:term_expansion} have been proven in \cite{PoSc19a}. Note that our specific choice of the integral projections \eqref{scrubs} lead to the classical ANOVA decomposition. A different choice of projection is also possible, e.g., the anchored variant used in the multivariate decomposition method, see \cite{KuNuPlSlWa17, GiKuNuWa18}.

In order to measure the importance of an ANOVA term $f_{\u}$ in relation to the function $f$, we use \emph{global sensitivity indices} or \emph{Sobol indices}
\begin{equation}\label{eq:gsi}
	\gsi{\u}{f} \coloneqq \frac{\va{f_{\u}}}{\va{f}}
	\quad\text{with}\quad
	\va{f} \coloneqq \norm{f}{\L_2(\T^d)}^2 - \abs{\fc{\b 0}{f}}^2,
\end{equation}
cf.~\cite{So90, So01, LiOw06}. This motivates the concept of effective dimensions. The superposition dimension is a particular notion of effective dimension and is given as \begin{equation}\label{superpos}
	d^{(\mathrm{sp})} = \min\left\{s \in \D\colon \sum_{\substack{\u \subset \D \\ \au \leq s}}\va{f_{\u}} \geq \alpha \va{f} \right\}
\end{equation} for an accuracy $\alpha \in [0,1]$. In other words, it is the smallest integer $s \in \D$ such that the variance of the function can mostly be explained by up to $s$-dimensional ANOVA terms. 

The number of ANOVA terms is $2^d$ and therefore grows exponentially in the dimension which reflects the well-known curse of dimensionality. In order to circumvent the curse, we use the idea of truncating the ANOVA decomposition and only taking a certain number of terms into account. A subset $U$ of the power set $\mathcal{P}(\D)$ is called \emph{subset of ANOVA terms} if for every $\u \in U$, all subsets $\v\subset\u$ are also elements of $U$. A special truncation possibility is motivated by the superposition dimension $d^{(\mathrm{sp})} \in \D$. We may form a set of terms \begin{equation}\label{eq:pizmosch}
	U_{d_s} \coloneqq \left\{ \u \subset \D \colon \au \leq d_s \right\}
\end{equation} where the order is limited by a superposition threshold $d_s \in \D$, see e.g.\ \cite{Holtz11,PoSc19a}. The condition that a subset of ANOVA terms $U \subset \mathcal{P}(\D)$ has to be downward closed can be relaxed of one assumes $f_{\u} \equiv 0$ for $\u \notin U$. 

For a subset of ANOVA terms $U \subset \mathcal{P}(\D)$, we define the truncation \begin{equation}\label{universe}
	\mathrm{T}_{U} f \coloneqq \sum_{\u\in U} f_{\u}.
\end{equation} and the special case $\mathrm{T}_{d_s} f \coloneqq \mathrm{T}_{U_{d_s}} f$. In \cite{PoSc19a} it has been shown that functions of certain types of smoothness, e.g.\ dominating-mixed smoothness, have a low superposition dimension and can therefore be approximated well by a truncation with $\mathrm{T}_{d_s} f$, see also \cite{Ow19}.

To obtain finite dimensions, we truncate the series expansion \eqref{eq:anova:term_expansion} to finite frequency sets which we introduce in the following. We start with the one-dimensional frequency set $ \mathcal I_N \coloneqq \mathbb Z\ \cap [-N/2,N/2)$, $N \in 2\mathbb N$, with $2\mathbb N$ being the even natural numbers. For $N \in 2\mathbb N$ and $\u \subset \D$ we define \begin{equation}\label{metallica}
	\mathcal I_{N}^{\bm u, d}
	\coloneqq \left\{
	\bm k\in \Z^d  \colon \supp \k = \u \text{ and } k_j \in \mathcal{I}_{N} \text{ for }  j \in \u
	\right\}.
\end{equation}
\begin{remark}\label{rem:cardinality}
	For a bandwidth $N \in 2\mathbb N$ we have $\{ {\b\ell}_{\u} \colon \b\ell \in \mathcal I_{N}^{\bm u, d} \} = (\mathcal I_{N}\setminus\{0\})^{\au}$ and, thus, we obtain for the cardinality
	$| \mathcal I_{N}^{\bm u, d} | = (N-1)^{\au}$.
\end{remark}

For a given subset of ANOVA terms $U \subset \mathcal{P}(\D)$ and bandwidths $N_{\u} \in 2\mathbb{N}$, $\u \in U$, we define a \emph{grouped index set} as the disjoint union \begin{equation}\label{groupedset}
	\mathcal{I}_{\b N}(U) \coloneqq \bigcup_{\bm u\in U} \mathcal I_{N_{\u}}^{\bm u, d}
\end{equation} with $\b N \coloneqq (N_{\u})_{\u \in U}$.
The corresponding Fourier partial sum is then given by
\begin{equation}\label{eq:rollingstones}
	S_{\mathcal{I}_{\b N}(U)} f  (\x) \coloneqq \sum_{\bm k\in\mathcal{I}_{\b N}(U)} c_{\bm k}(f) \,\mathrm e^{2\pi\mathrm i\langle\bm k,\bm x\rangle}.
\end{equation}
\begin{remark}
	Setting $N_{\u} = N \in 2\mathbb N$ for a bandwidth $N \in 2\N$ and $\u \subset \D$, we have a disjoint union of the hypercube \begin{equation}\label{hypercube}
		\mathcal{I}_{N}^d = \mathcal{I}_{(N, N, \dots, N)}(\mathcal{P}(\D)) =  \bigcup_{\bm u\subset\D} \mathcal I_{N}^{\bm u, d}.
	\end{equation}
	This identity is visualized in \cref{fig:dropkickmurphys}.
\end{remark}
\begin{figure}
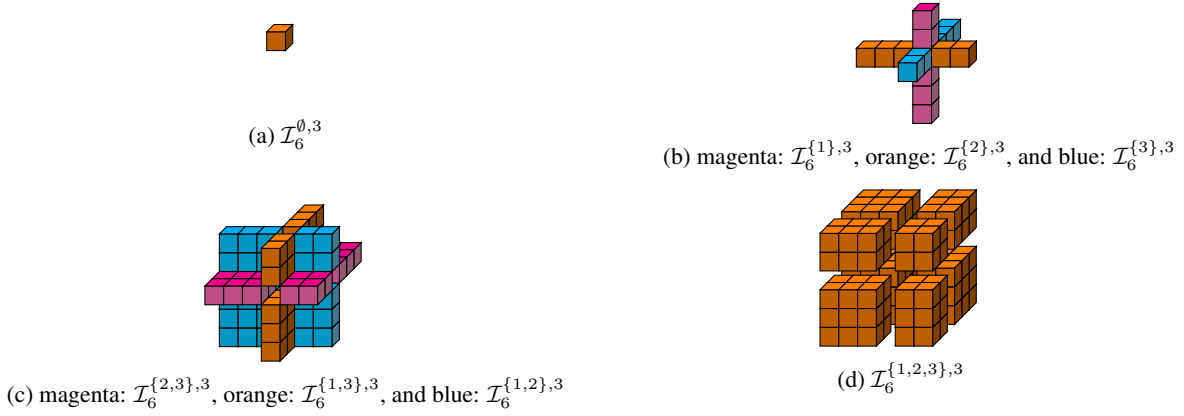

  \begin{subfigure}{0.49\textwidth}
    \centering
    \drawfrequencyset{data/frequencyset0.dat}
    \vspace{0.7cm}
    \caption{
      $\mathcal I_6^{\emptyset, 3}$
    }
  \end{subfigure}
  \begin{subfigure}{0.49\textwidth}
    \centering
    \drawfrequencyset{data/frequencyset1.dat}
    \caption{
      magenta: $\mathcal I_{6}^{\{1\}, 3}$,
      orange: $\mathcal I_{6}^{\{2\}, 3}$,
      and blue: $\mathcal I_{6}^{\{3\}, 3}$
    }
  \end{subfigure}
  \begin{subfigure}{0.49\textwidth}
    \centering
    \vspace{0.4cm}
    \drawfrequencyset{data/frequencyset2.dat}
    \caption{
      magenta: $\mathcal I_{6}^{\{2, 3\}, 3}$,
      orange: $\mathcal I_{6}^{\{1, 3\}, 3}$,
      and blue: $\mathcal I_{6}^{\{1, 2\}, 3}$
    }
  \end{subfigure}
  \begin{subfigure}{0.49\textwidth}
    \centering
    \drawfrequencyset{data/frequencyset3.dat}
    \caption{
      $\mathcal I_{6}^{\{1, 2, 3\}, 3}$
    }
  \end{subfigure}
  \caption{Decomposition of $\mathcal{I}_{6}^3$ into its groups.}
  \label{fig:dropkickmurphys}
\end{figure}

\section{Grouped Fourier Transform}\label{sec:gft}

The \emph{nonequispaced discrete Fourier transformation} in $d \in \mathbb N$ dimensions is given by
\begin{equation*} %\label{eq:nirvana}
	f_{\bm x}  = \sum_{\bm k\in\mathcal{I}_{N}^d} \hat f_{\bm k} \, \e^{2\pi\i\langle\k,\x\rangle}
\end{equation*}
for $\bm x\in\mathcal X$, where $\mathcal X \subset\mathbb T^d$ is an arbitrary finite set of nodes on the torus $\mathbb{T}$. Moreover, we have the index set $\mathcal{I}_{N}^d \subset \mathbb Z^d$ from \eqref{hypercube}, coefficients $\hat f_{\bm k} \in \C$, $\bm k \in \mathcal{I}_{N}^d$, and $\langle\cdot,\cdot\rangle$ the Euclidean inner product, see \cite[Chapter 7]{PlPoStTa18}.
The related Fourier matrix is given by
\begin{equation}\label{eq:FM}
	\bm F\left(\mathcal X, \mathcal I\right)
	= \left(   \e^{2\pi\i\langle\k,\x\rangle} \right)_{\bm x\in\mathcal X, \bm k \in \mathcal I}
	\in\mathbb C^{|\mathcal X| \times |\mathcal I|},
\end{equation}
where a straightforward matrix-vector multiplication requires ${\cal O}(|\mathcal X|\,|\mathcal I|)$ arithmetical operations.

In the following, we introduce the \textit{grouped Fourier transform} as a special case with the frequency index sets $\mathcal{I}_{\b N}(U)$, $U \subset \mathcal{P}(\D)$, $\b N \in (2\N)^{\abs{U}}$, as in \eqref{groupedset}, i.e., they consist of a disjoint union of lower dimensional frequency index sets with zeros along some dimensions. Index sets of this type were also used in \cite{PoSc19a}.
In this section, we discuss all necessities to carry out the computations in lower dimensions instead of the high spatial dimension.
This will peak in the identity \eqref{eq:somafm}, which is then used to develop fast algorithms.

\begin{definition}\label{def:skatoons}
	Let $U$ be a subset of $\mathcal P(\D)$, where $\mathcal P$ stands for the powerset, and $\mathcal{I}(U)$ a grouped index set of structure \eqref{groupedset} with $N_{\u}\in 2\N$, $\bm u\in U$, bandwidth parameters, and $\hat f_{\bm k}\in\mathbb C$ coefficients for $\bm k\in \mathcal{I}_{\b N}(U)$.
	On an arbitrary set of nodes $\bm x\in\mathcal X$, the \emph{grouped Fourier transform} is then given by
	$$
	f_{\bm x} = \sum_{\bm k\in\mathcal{I}_{\b N}(U)} \hat f_{\bm k} \, \e^{2\pi\i\langle\k,\x\rangle}
	= \sum_{\bm u\in U}\sum_{\bm k\in \mathcal I_{N_{\u}}^{\bm u, d}} \hat f_{\bm k}  \, \e^{2\pi\i\langle\k,\x\rangle}.
	$$
\end{definition}

Computing the grouped Fourier transform from \cref{def:skatoons} via the naive approach, i.e., setting up the Fourier matrix $\bm F(\mathcal X, \mathcal{I}_{\b N}(U))$ and performing a matrix-vector multiplication with $\bm{\hat f} = (\bm{\hat f}(\bm u))_{\bm u \in U}$ and $\bm{\hat f}(\bm u) = (\hat f_{\bm k})_{\k\in \mathcal{I}_{N_{\u}}^{\bm u, d}}$ results in the complexity class
$$
\mathcal O
\left(|\mathcal X|\sum_{\bm u\in U}|\mathcal I_{N_{\u}}^{\bm u, d}|\right)
=
\mathcal O
\left(|\mathcal X|\sum_{\bm u\in U}(N_{\u}-1)^{\au}\right),
$$ cf.\ \cref{rem:cardinality}

This is infeasible in high dimensions, but can be reduced as follows.
First, we observe that we are working with a block matrix, i.e.,
\begin{equation}\label{bananarama}
	\bm F(\mathcal X, \mathcal{I}_{\b N}(U)) = \begin{pmatrix}
		\bm F(\mathcal X, \mathcal I_{N_{\u_1}}^{\u_1, d}) \quad \bm F(\mathcal X, \mathcal I_{N_{\u_2}}^{\u_2, d}) \quad  \cdots\quad \bm F(\mathcal X, \mathcal I_{N_{\u_{\abs{U}}}}^{\u_{\abs{U}}, d})
	\end{pmatrix}
\end{equation} when introducing an order on the sets $\u$.
Using the sparse structure of $\mathcal I_{N_{\u}}^{\bm u, d}$, we obtain
\begin{align*}
	\bm F\left(\mathcal X, \mathcal I_{N_{\u}}^{\bm u, d}\right)
	& = \left(  \e^{2\pi\i\langle\k,\x\rangle} \right)_{\bm x\in\mathcal X, \bm k \in \mathcal I_{N_{\u}}^{\bm u, d}} \\
	& = \left(  \e^{2\pi\i\langle\b\ell,\x_{\u}\rangle} \right)_{\bm x\in\mathcal X, \b\ell \in (\mathcal I_{N_{\u}}\setminus\{0\})^{\au}}.
\end{align*}
For $\bm{\hat f}(\bm u)$ with frequency domain 
$\mathcal I_{N_{\u}}^{\bm u, d}$ we now introduce an extension such that the coefficients corresponding to zero-frequencies are filled with zeros and the frequency-dimension is reduced to $|\bm u|$.
Thus, we have frequencies in $\mathcal I_{N_{\u}}^{|\bm u|}$.
This extension is visualized in \cref{fig:extension} and can be achieved formally by defining a vector $\bm{\hat g}(\bm u) \in \C^{|\mathcal I_{N_{\u}}^{\bm u, d}|}$ with entries \begin{equation*}
	\left(\bm{\hat g}(\bm u)\right)_{\b\ell \in \mathcal I_{N_{\u}}^{|\bm u|}} = \begin{cases}
		\hat f_{\b h} &\colon \supp\b\ell = \{1,2,\dots,\au\}\\
		0 &\colon \text{otherwise} 
	\end{cases}
\end{equation*} with $\b h \in \Z^d$ the frequency such that $\b{h}_{\u} = \b\ell$ and $\b{h}_{\uc} = \b 0$.
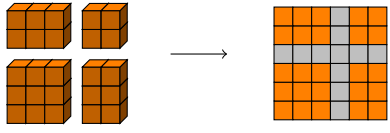
\begin{figure}
  \centering
  \begin{tikzpicture}[scale = 0.25]
    \CatchFileDef\loadeddata{data/extension.dat}{\endlinechar=-1}
    \foreach \x\y\z\color in \loadeddata{
      \draw[fill=black!25!\color] (\x-0.5,\y-0.5,\z+0.5) -- (\x+0.5,\y-0.5,\z+0.5) -- (\x+0.5,\y+0.5,\z+0.5) -- (\x-0.5,\y+0.5,\z+0.5) -- cycle;
      \draw[fill=black!50!\color] (\x+0.5,\y-0.5,\z+0.5) -- (\x+0.5,\y-0.5,\z-0.5) -- (\x+0.5,\y+0.5,\z-0.5) -- (\x+0.5,\y+0.5,\z+0.5) -- cycle;
      \draw[fill=\color] (\x+0.5,\y+0.5,\z+0.5) -- (\x+0.5,\y+0.5,\z-0.5) -- (\x-0.5,\y+0.5,\z-0.5) -- (\x-0.5,\y+0.5,\z+0.5) -- cycle;
    }
    \foreach \x\y\z\color in \loadeddata{
      \draw[fill=\color] (\x-0.5+14,\y-0.5) -- (\x+0.5+14,\y-0.5) -- (\x+0.5+14,\y+0.5) -- (\x-0.5+14,\y+0.5) -- cycle;
    }
    \foreach \x\y\z\color in {0/-3, 0/-2, 0/-1, 0/0, 0/1, 0/2, -3/0, -2/0, -1/0, 1/0, 2/0}{
      \draw[fill=black!25!white] (\x-0.5+14,\y-0.5) -- (\x+0.5+14,\y-0.5) -- (\x+0.5+14,\y+0.5) -- (\x-0.5+14,\y+0.5) -- cycle;
    }
    \draw[->] (5,0) -- (8, 0);
  \end{tikzpicture}
  \caption{Visualization of the extension for $\mathcal I_{6}^{\{2, 3\}, 3}$ to $\mathcal I_{6}^{2}$.}
  \label{fig:extension}
\end{figure}

We obtain the identity
\begin{equation}\label{eq:somafm}
	\bm F\left(\mathcal X, \mathcal I_{N_{\u}}^{\bm u, d}\right) \bm{\hat f}(\bm u)
	=
	\bm F\left(\mathcal X_{\bm u}, \mathcal I_{N_{\u}}^{|\bm u|}\right) \bm{\hat g}(\bm u),
\end{equation}
with $\mathcal X_{\bm u} \coloneqq \{ \x_{\u} \colon \x \in \mathcal X \}$ where the left-hand side operates in $d$ dimensions and the right-hand side in $|\bm u|$ dimensions. Note that $\mathcal X_{\bm u}$ is a multiset, i.e., it may contain duplicates.
Furthermore, the right-hand side of \eqref{eq:somafm} can be computed via a nonequispaced fast Fourier transform (NFFT), see \cite{KeKuPo09, PlPoStTa18, nfft3}, in the complexity $\mathcal O(N_{\u}^{|\bm u|}\log N_{\u}+|\log\epsilon|^{|\bm u|}|\mathcal X|)$ for precision $\epsilon$. In summary, we have \begin{equation*}
	\bm F(\mathcal X, \mathcal{I}_{\b N}(U))\bm{\hat f} = \sum_{\u \in U} \bm F\left(\mathcal X, \mathcal I_{N_{\u}}^{\bm u, d}\right) \bm{\hat f}(\bm u) = \sum_{\u \in U} \bm F\left(\mathcal X_{\bm u}, \mathcal I_{N_{\u}}^{|\bm u|}\right) \bm{\hat g}(\bm u)
\end{equation*} which means that we have reduced one $d$-dimensional NFFT to $\abs{U}$ low-dimensional NFFTs. While the curse of dimensionality effectively hinders fast multiplication with $\bm F(\mathcal X, \mathcal{I}_{\b N}(U))\bm{\hat f}$ in $d$ dimensions we only need up to $d_s$-dimensional NFFTs for $U \subseteq U_{d_s}$ and the number of terms in $U_{d_s}$ grows only polynomial in the spatial dimension $d$ for a fixed $d_s$, see \eqref{superpos}. Moreover, the low-dimensional transformations can be computed independent of each other which allows for parallelization. Given enough resources, it would be possible to compute multi-threaded variants of each single transformation simultaneously which yields a high benefit in execution time. It follows that the multiplication with $\bm F$ can be accomplished in
$$
\mathcal O\left(
\sum_{\bm u\in U}
\left(N_{\u}^{|\bm u|}\log N_{\u}+|\log\epsilon|^{|\bm u|}|\mathcal X|\right)
\right)
$$
by the NFFT for precision $\epsilon$.
\begin{remark}
	For an error estimate we have for the $\text{NFFT}$
	$$
	\left| \sum_{\bm k\in \mathcal I_{N_{\u}}^{\bm u, d}} \hat f_{\bm k} \mathrm e^{2\pi\mathrm i\langle\bm k,\bm x\rangle}
	- \text{NFFT}\left(\bm{\hat g}(\bm u)\right) \right|
	\le C_{\bm u}(N_{\u}) \left\|\bm{\hat f}(\bm u)\right\|_1
	$$
	where $C_{\bm u}(N_{\u})$ decays exponentialy in $N_{\u}$, cf.~\cite[Theorem 7.8]{PlPoStTa18} and references therein.
	Since the grouped Fourier transform is a sum of the above, we can bound the overall error by
	$$
	\left(\max_{\bm u\in U} C_{\bm u}(N_{\u})\right) \left\|\bm{\hat f}\right\|_1.
	$$
\end{remark}

The adjoint problem, i.e., the multiplication of $\bm F(\mathcal X, \mathcal{I}_{\b N}(U))\herm$ with a vector $\b f \in \C^{\abs{\mathcal{X}}}$ can be realized in the same fashion. We have \begin{equation*}
	\bm F(\mathcal X, \mathcal{I}_{\b N}(U))\herm \b f = \begin{pmatrix}
		\bm F\left(\mathcal X, \mathcal I_{N_{\u_1}}^{\u_1, d}\right)\herm \b f \\
		\vdots \\
		\bm F\left(\mathcal X, \mathcal I_{N_{\u_{\abs{U}}}}^{\u_{\abs{U}}, d}\right)\herm \b f 
	\end{pmatrix} = \begin{pmatrix}
		\b F\left(\mathcal X_{\u_1}, \mathcal I_{N_{\u_1}}^{|\u_1|}\right)\herm \b f \\
		\vdots \\
		\b F\left(\mathcal X_{\u_{\abs{U}}}, \mathcal I_{N_{\u_{\abs{U}}}}^{|\u_{\abs{U}}|}\right)\herm \b f
	\end{pmatrix}
\end{equation*} which yields the same benefits. We can decompose our $d$-dimensional adjoint transform in $\abs{U}$ low-dimensional transforms and are able to compute all of them simultaneously by using parallelization.  

\begin{remark}\label{rem:extension}
	The considerations in this section are not limited to the exponential functions, but rather, can be extended to other complete orthonormal systems.
	In fact, for $\exp(2\pi\mathrm i\langle \bm k, \cdot\rangle)$ and $\cos(\pi \langle\bm k,\cdot\rangle)$ the presented ideas are already implemented in the julia package \texttt{GroupedTransforms} which can be found on GitHub, see \path{github.com/NFFT/GroupedTransforms}, where the nonequispaced fast cosine transform (NFCT) is used for performance in the latter. We refer to the references \cite{PoSc19b, PoSc21}.
\end{remark}

\section{Approximation}
\label{sec:approx}

In this section, we consider the problem of approximating periodic functions $\fun{f}{\T^d}{\C}$ with high spatial dimension $d \in \N$. In particular, we have a scattered data setting where the given data about the unknown function $f$ are sampling values $\bm y = (f(\bm x)+\eta)_{\bm x\in\mathcal X}$ which may contain noise $\eta$ on a finite set of nodes $\mathcal X\subset\mathbb T^d$. We are looking for an approximation with a Fourier partial sum \eqref{eq:rollingstones} of the truncated ANOVA decomposition and want to incorporate the fast algorithms from \cref{sec:gft}. We assume that the function has a low superposition dimension \eqref{superpos}. This can occur either because it belongs to a function class that has this property, see e.g.\ \cite{PoSc19a}, or it stems from a real world application with sparsity-of-effects. 

We use the method proposed in \cite{PoSc19a} for approximation, which suggests to choose a superposition threshold $d_s \in \D$. In particular, we will use the ANOVA terms $f_{\u}$, $\u \in U_{d_s}$, from \eqref{eq:pizmosch} with grouped frequency index set $\mathcal{I}_{\b N}(U_{d_s})$, see \eqref{groupedset}, and bandwidth parameters $N_{\u} \in 2\N$.
Since the Fourier coefficients $\fc{\k}{f} \in \C$, $\k \in \mathcal{I}_{\b N}(U_{d_s})$, are not known, we are going to approximate them by coefficients $\hat{f}_{\k} \approx \fc{\k}{f}$ that we determine from the given data $\mathcal X$ and $\y$. We obtain an approximate Fourier partial sum
\begin{equation*}
	S_{\mathcal{I}_{\b N}(U_{d_s})}^{\mathcal X} f (\x) \coloneqq \Fseries{\k}{\mathcal{I}_{\b N}(U_{d_s})}{\hat{f}_{\k}}{\x}
\end{equation*} 
with $\hat{f}_{\k} \in \C$ such that
$
S_{\mathcal{I}_{\b N}(U_{d_s})}^{\mathcal X} f (\x) 
\approx
S_{\mathcal{I}_{\b N}(U_{d_s})} f (\x) 
$. 

For any subset of ANOVA terms $U \subset U_{d_s}$ (with equality possible) computing the approximate coefficients $\hat{f}_{\k}$ is done via solving the least-squares problem \begin{equation}\label{min_prob}
	\min_{\bm{\hat f}} \frac{1}{2}\norm{\y - \bm F(\mathcal X, \mathcal{I}_{\b N}(U)) \hat{\b f}}{2}^2
\end{equation} with grouped index set $\mathcal{I}_{\b N}(U)$, bandwidth parameters $N_{\u} \in 2\N$, $\u \in U$, and Fourier matrix \eqref{eq:FM}. For $\abs{\mathcal X} > \abs{\mathcal{I}_{\b N}(U)}$ and $\bm F(\mathcal X, \mathcal{I}_{\b N}(U))$ full rank, the system has a unique solution. Moreover, if $\mathcal X$ are uniformly distributed i.i.d.\ nodes and the oversampling factor is large enough, the matrix has full rank with high probability and we get an estimate for the approximation error, cf. \cite[Theorem 6.7]{PoSc19a} and \cite[Section 5.2]{KaeUlVo19}. 

In \cref{sec:lsqr} we use the iterative LSQR method to solve \eqref{min_prob}. However, we are going to extend this by adding an $\ell_2$ regularization term known from the ordinary Tikhonov regularization and simultaneously consider weights which incorporate information about the decay rate of the Fourier coefficients, i.e., the smoothness of the function. In particular, we use weights corresponding to functions from a Sobolev type space
\begin{equation*}
	\sobolev{\omega}(\T^d) \coloneqq \left\{ f \in \L_2(\T^d) \colon \norm{f}{\sobolev{\omega}(\T^d)} \coloneqq \sqrt{\sum_{\k \in \Z^d}\omega^2(\k)\abs{\fc{\k}{f}}^2} < \infty \right\}
\end{equation*}
for a weight function $\fun{\omega}{\Z^d}{[1,\infty)}$, see e.g.~\cite{KuMaUl16, PoSc19a}. We use sensitivity analysis on $S_{\mathcal{I}_{\b N}(U_{d_s})}^{\mathcal X} f (\x)$ to detect which of the ANOVA terms $f_{\bm u}$ for $\u \in U_{d_s}$ are of a high-importance to the function and determine an active set \begin{equation}\label{activeset}
	U^\ast = U_{\mathcal X,\y}^{(\b\eps)} \coloneqq \emptyset \cup \{ \u \subset \D \colon \gsi{\u}{S_{\mathcal{I}_{\b N}(U_{d_s})}^{\mathcal X} f} > \varepsilon_{\au} \}
\end{equation} with threshold vector $\b\eps \in (0,1)^{d_s}$. Solving the problem \eqref{min_prob} with $\mathcal{I}_{\b N}(U^\ast)$ and new bandwidth parameters $N_{\u} \in 2\N$, $\u \in U^\ast$, yields the approximation $S_{\mathcal{I}_{\b N}(U^\ast)}^{\mathcal X} f (\x)$. 

Our goal is to compare how $\ell_2$ regularization impacts the addition of noise which was not considered in \cite{PoSc19a} and has been formulated as an open problem. Moreover, we discuss the addition of known Sobolev smoothness information leading to a smaller \textit{search space} which can be seen by using the equivalent Ivanov formulation of the regularization functional, cf.\,\cite{ORA16}, where we minimize the data fitting term subject to a bound on the Sobolev norm. This allows for an oversampling factor $\abs{X}/\abs{\mathcal{I}_{\b N}(U)}$ close to or smaller than one. In this case, we will even be in an undetermined setting, i.e., $\abs{X} < \abs{\mathcal{I}_{\b N}(U)}$.

In \cref{sec:fista} we use a regularization technique to introduce an approach that combines the two steps from before, i.e., approximation and active set detection. The method is a variation of the group lasso technique from \cite{YL06} that promotes sparsity in defined groups. In our case, the groups will be the ANOVA terms or subsets of coordinate indices $\u \in U_{d_s}$. As in the prior case, we add Sobolev weights which will be used to incorporate smoothness information about the function. 

Both approaches rely on iterative solvers which require fast multiplication with $\bm F(\mathcal X, \mathcal{I}_{\b N}(U))$ and its adjoint which is provided by the grouped Fourier transform from \cref{sec:gft}. As a quality measure we use the approximation error
\begin{equation}\label{L2error}
	\eps(\mathcal{X},\mathcal{I}_{\b N}(U)) \coloneqq\frac{1}{\norm{f}{\L_2(\T^d)}} \norm{f-S_{\mathcal{I}_{\b N}(U)}^{\mathcal X} f }{\L_2(\T^d)}
\end{equation} which can be computed if the norm and the Fourier coefficients of the function are known for the purpose of computing the error via Parseval's identity \begin{equation*}
	\norm{f-S_{\mathcal{I}_{\b N}(U)}^{\mathcal X} f }{\L_2(\T^d)}^2 = \norm{f}{\L_2(\T^d)}^2 + \sum_{\k \in \mathcal{I}_{\b N}(U)} \abs{\hat{f}_{\k}-\fc{\k}{f}}^2 -\sum_{\k \in \mathcal{I}_{\b N}(U)} \abs{\fc{\k}{f}}^2.
\end{equation*}

For comparability reasons, we work with the same test function throughout this section
\begin{multline}\label{eq:testfun}
	f:\T^9 \to \mathbb R, \\
	\bm x \mapsto B_2(x_1) B_4(x_3) B_6(x_8)+  B_2(x_2) B_4(x_5) B_6(x_6) + B_2(x_4) B_4(x_7) B_6(x_9),
\end{multline}
where $B_2$, $B_4$ and $B_6$ are parts of univariate, shifted, scaled, and dilated B-splines of order 2, 4, and 6, respectively. Their Fourier series is given by 
$$
B_j(x) \coloneqq c_j \sum_{k \in \Z} \mathrm{sinc}^j\left(\frac{\pi\cdot k}{j}\right) \cos( \pi \cdot k ) \,\e^{2\pi\i k \cdot x}
\quad\text{for}\quad  j = 2,4,6,
$$
with $\mathrm{sinc}(x) \coloneqq \sin(x)/x$ and the three normalization constants $c_2 \coloneqq \sqrt{3/4}$, $c_4 \coloneqq \sqrt{315/604}$, $c_6 \coloneqq \sqrt{277200/655177}$ such that $\norm{B_j}{\L_2(\T)} = 1$, $j = 2,4,6$.

The test function works well for computing the error \eqref{L2error} since we have the Fourier coefficients $\fc{\k}{f}$ and the norm $\norm{f}{\L_2(\T^d)}$ explicitly given. Moreover, The function $f$ has superposition dimension $d^{(\mathrm{sp})} = 3$, see \eqref{superpos}, for arbitrary high accuracy, i.e., it can be represented by at most three-dimensional ANOVA terms with $f = \mathrm{T}_3 f$. This leads to $d_s = 3$ being the optimal choice for the superposition threshold with no error caused by the ANOVA truncation. We have the active set of terms \begin{equation*}
	\u \in U^\ast \coloneqq \mathcal{P}(\{ 1,3,8 \}) \cup \mathcal{P}(\{ 2,5,6 \}) \cup \mathcal{P}(\{ 4,7,9 \})
\end{equation*} with $f_{\u} = 0$ for $\u \notin U^\ast$.
The function also has dominating-mixed smoothness of $3/2 - \eps$ for every $\eps > 0$, i.e., $f \in \sobolev{\omega_\eps}(\T^9)$ with 
$$
\omega_\eps(\k) = \prod_{j \in \supp \k} (1+\abs{k_j})^{\frac{3}{2}-\eps},
$$
cf.~\cite{PoVo14}. 

The ANOVA terms $f_{\u}$ can be computed analytically such that we obtain \begin{align}
	f_\emptyset &= 3 \prod_{j\in\{2,4,6\}} c_j \nonumber\\
	f_{\{i\}}(x_i) &= \frac{\prod_{j\in\{2,4,6\}} c_j}{c_{\mathrm{o}(i)}} \left(B_{\mathrm{o}(i)}(x_i) - c_{\mathrm{o}(i)}\right) \label{onedimterm}
\end{align} for the constant and the one-dimensional terms with $i = 1,2,\dots,9$ and \begin{equation*}
	\mathrm{o}(i) \coloneqq \begin{cases}
		2 \quad &\colon i \in \{1,2,4\} \\
		4 &\colon i \in \{3,5,7\} \\
		6 &\colon i \in \{8,6,9\}.
	\end{cases}
\end{equation*} We find for the two-dimensional terms $f_{\u}$, $\u = \{i,j\}$, $i,j\in\{1,2,\dots,9\}$, that $f_{\{i,j\}} \equiv 0$ for $\u\notin U^\ast$ and \begin{equation}\label{twodimterm}
	f_{\{i,j\}}(x_{\{i,j\}}) = \frac{\prod_{j\in\{2,4,6\}} c_j}{c_{\mathrm{o}(i)}c_{\mathrm{o}(j)}} \left(B_{\mathrm{o}(i)}(x_i) - c_{\mathrm{o}(i)}\right) \left(B_{\mathrm{o}(j)}(x_j) - c_{\mathrm{o}(j)}\right)
\end{equation} for $\u\in U^\ast$. Finally, we get for the three-dimensional terms $f_{\u}$, $\u = \{i,j,\ell\}$, $i,j,\ell\in\{1,2,\dots,9\}$, that $f_{\u} \equiv 0$ again for $\u\notin U^\ast$ and \begin{equation}\label{threedimterm}
	f_{\{i,j,\ell\}}(x_{\{i,j,\ell\}})= \left(B_{\mathrm{o}(i)}(x_i) - c_{\mathrm{o}(i)}\right) \left(B_{\mathrm{o}(j)}(x_j) - c_{\mathrm{o}(j)}\right) \left(B_{\mathrm{o}(\ell)}(x_\ell) - c_{\mathrm{o}(\ell)}\right)
\end{equation} for $\u\in U^\ast$. However, in the approximation scenario we only have the data about the function $f$ and not $f$ itself. The norm and the exact Fourier coefficients are only used to compute the error $\eps(\mathcal{X},\mathcal{I}_{\b N}(U))$ and compare the order of the sensitivity indices.

\subsection{LSQR}\label{sec:lsqr}

In this section, we discuss the addition of an $\ell_2$ regularization the problem \eqref{min_prob} and the incorporation of a-priori information about the decay of the Fourier coefficients. We characterize the smoothness of a function by the decay of the Fourier coefficients, i.e., a function $f \in \sobolev{\omega}(\T^d)$ has the smoothness defined by the weight $\fun{\omega}{\Z^d}{[1,\infty)}$. This directly implies for the decay $\abs{\fc{\k}{f}} \in \mathrm{o}(\omega^{-1}(\k))$ since $\omega^2(\k)\abs{\fc{\k}{f}}^2 \rightarrow 0$ for $\norm{\k}{\infty} \rightarrow \infty$ by definition. In order to incorporate this information we add a term to \eqref{min_prob} and obtain the modified problem 
\begin{equation}\label{new_min_prob}
	\min_{\hat{\b f}} \frac{1}{2}\norm{\y - \F(\mathcal X, \mathcal{I}_{\b N}(U)) \hat{\b f}}{2}^2 + \lambda \norm{\hat{\b f}}{\bm{\hat W}}^2
\end{equation} with regularization parameter $\lambda > 0$, weighted norm $\|\hat{\b f}\|_{\bm{\hat W}}^2 = \hat{\b f}\herm \bm{\hat W}\hat{\b f}$ and $\bm{\hat W} = \diag(\omega(\k))_{\k\in \mathcal I}$ a diagonal matrix. This can be the identity if no smoothness information is known and $\omega(\k) \equiv 1$, i.e., all frequencies are penalized equally. Problem \eqref{new_min_prob} always has a unique solution in our setting which we aim to find by applying the iterative LSQR algorithm, see \cite{PaSa82}. 

We apply the method with our newly obtained minimization problem \eqref{new_min_prob} to the test function \eqref{eq:testfun} which we sample at $\mathcal X \subset \T^d$, $\abs{\mathcal X} = 10\,000$, uniform i.i.d.\,random nodes and use the Sobolev weights
\begin{equation}\label{sobolevweights}
	\omega_s(\k)  = \prod_{j=1}^9 (1+|k_j|)^s 
\end{equation}
for $s \geq 0$ in $\bm{\hat W}$. In the following numerical experiments we use different parameters $s$ and $N_{\u}$ as well as $\lambda \in [\e^0,\e^{10}]$. The results are visualized with the $\L_2$-error \eqref{L2error} for the approximation $S_{\mathcal{I}_{\b N}(U^\ast)}^{\mathcal X} f$ and the global sensitivity indices $\gsi{\u}{S_{\mathcal{I}_{\b N}(U_3)}^{\mathcal X} f}$, cf.\ \eqref{eq:gsi}, for each $|\bm u|=1,2,3$ separately. Note that the errors and global sensitivity indices have been averaged over 100 random draws of the node set $\mathcal X$. 

\begin{enumerate}[(i)]
	\item
	In the first setting we choose $\omega(\bm k) = 1$ for all $\bm k\in\mathcal I$, which penalizes all frequencies equally. The regularization term in \eqref{new_min_prob} then becomes a non-weighted $\ell_2$ regularization without smoothness information.
	To compensate for this lack of information, we choose small bandwidths
	\begin{equation}\label{eq:smallbandwidths}
		N_{\u}
		= \begin{cases}
			26 & \text{for all } \u \text{ with } |\bm u| = 1 \\
			6 & \text{for all } \u \text{ with }  |\bm u| = 2 \\
			4 & \text{for all } \u \text{ with }  |\bm u| = 3,
		\end{cases}
	\end{equation}
	which add up to $3\,394$ frequencies such that we are in an overdetermined setting. The results of solving the minimization is depicted in \cref{fig:lsqrs=0sigma=0.0}. We are able to clearly detect the active set $U^\ast$ and distinguish the global sensitivity indices of the contained terms from the rest. In (b) we observe that the global sensitivity indices $\rho(\bm u,S_{\mathcal{I}_{\b N}(U_3)}^{\mathcal X} f)$ are larger than zero for all nine one-dimensional ANOVA terms $f_{\bm u}$, $\u \in U^\ast$, $\au = 1,$ occuring in the test function. We see that the sensitivity indices are on three different levels. They relate to the smoothness if we consider the analytical term \eqref{onedimterm}, i.e., depending on which of the three splines is involved, the corresponding ANOVA term is more smooth if the spline is of a higher order. The nine two-dimensional terms $f_{\bm u}$, $\u \in U^\ast$, $\au = 2,$ are clearly separated from the inactive terms in $\mathcal{P}(\D)\setminus U^\ast$, see (c). The sensitivity indices here are also grouped around three levels since in \eqref{twodimterm} there is always a product of two different splines yielding $\binom{3}{2} = 3$ types of smoothness for those terms. The plot (d) shows that we can also distinguish the three active terms $f_{\bm u}$, $\u \in U^\ast$, $\au = 3$, which are of the same smoothness since they are always a product of the three involved splines, see \eqref{threedimterm}. The minimal relative $\L_2$-error is $\eps(\mathcal{X},\mathcal{I}_{\b N}(U^\ast)) \approx 9.1\cdot 10^{-2}$ when using the active set of terms $U^\ast$ after sensitivity analysis. This experiment provides a comparison for the addition of noise and smoothness information.
	\begin{figure}
  \centering
  
  \begin{subfigure}[t]{0.24\linewidth}\centering
    \raggedleft
    \begin{tikzpicture}\begin{axis}[
      scale only axis, width = 2.1cm, height = 2.1cm,
      enlarge x limits = 0,
      xmode = log,
      xlabel = $\lambda$,
      ymin = 0, ymax = 1,
    ]
      \addplot[no marks] table[x = lambda, y = L2error] {data/lsqrs=0sigma=0.0.csv};
    \end{axis}\end{tikzpicture}
    \caption{$\eps(\mathcal{X},\mathcal{I}_{\b N}(U^\ast))$}
  \end{subfigure}
  \begin{subfigure}[t]{0.24\linewidth}\centering
    \raggedleft
    \begin{tikzpicture}\begin{axis}[
      scale only axis, width = 2.1cm, height = 2.1cm,
      enlarge x limits = 0,
      xmode = log,
      xlabel = $\lambda$,
      ymin = -0.005, ymax = 0.095,
    ]
      \foreach \u in {1, 2, 3, 4, 5, 6, 7, 8, 9}{
        \addplot[no marks, color = orange] table[x = lambda, y = \u]{data/lsqrs=0sigma=0.0.csv};
      }
    \end{axis}\end{tikzpicture}
    \caption{$\rho(\bm u, S_{\mathcal{I}_{\b N}(U_3)}^{\mathcal X} f)$ with $|\bm u| = 1$}
  \end{subfigure}
  \begin{subfigure}[t]{0.24\linewidth}\centering
    \raggedleft
    \begin{tikzpicture}\begin{axis}[
      scale only axis, width = 2.1cm, height = 2.1cm,
      enlarge x limits = 0,
      xmode = log,
      xlabel = $\lambda$,
      ymin = -0.005, ymax = 0.09,
    ]
      \foreach \u in {12, 14, 15, 16, 17, 19, 23, 24, 27, 28, 29, 34, 35, 36, 37, 39, 45, 46, 48, 57, 58, 59, 67, 68, 69, 78, 89}{
        \addplot[no marks, dash pattern = on 2pt off 1pt] table[x = lambda, y = \u]{data/lsqrs=0sigma=0.0.csv};
      }
      \foreach \u in {13, 18, 25, 26, 38, 47, 49, 56, 79}{
        \addplot[no marks, orange] table[x = lambda, y = \u] {data/lsqrs=0sigma=0.0.csv};
      }
    \end{axis}\end{tikzpicture}
    \caption{$\rho(\bm u, S_{\mathcal{I}_{\b N}(U_3)}^{\mathcal X} f)$ with $|\bm u| = 2$}
  \end{subfigure}
  \begin{subfigure}[t]{0.24\linewidth}\centering
    \raggedleft
    \begin{tikzpicture}\begin{axis}[
      scale only axis, width = 2.1cm, height = 2.1cm,
      enlarge x limits = 0,
      xmode = log,
      xlabel = $\lambda$,
      ymin = -0.005, ymax = 0.03,
    ]
      \foreach \u in {123, 124, 125, 126, 127, 128, 129, 134, 135, 136, 137, 139, 145, 146, 147, 148, 149, 156, 157, 158, 159, 167, 168, 169, 178, 179, 189, 234, 235, 236, 237, 238, 239, 245, 246, 247, 248, 249, 257, 258, 259, 267, 268, 269, 278, 279, 289, 345, 346, 347, 348, 349, 356, 357, 358, 359, 367, 368, 369, 378, 379, 389, 456, 457, 458, 459, 467, 468, 469, 478, 489, 567, 568, 569, 578, 579, 589, 678, 679, 689, 789}{
        \addplot[no marks, dash pattern = on 2pt off 1pt] table[x = lambda, y = \u]{data/lsqrs=0sigma=0.0.csv};
      }
      \foreach \u in {138, 256, 479}{
        \addplot[no marks, orange] table[x = lambda, y = \u] {data/lsqrs=0sigma=0.0.csv};
      }
    \end{axis}\end{tikzpicture}
    \caption{$\rho(\bm u, S_{\mathcal{I}_{\b N}(U_3)}^{\mathcal X} f)$ with $|\bm u| = 3$}
  \end{subfigure}
  \caption{LSQR with small bandwidths given in \eqref{eq:smallbandwidths} on exact data with $s=0$ (orange: active ANOVA terms $U^\star$, dashed: inactive ANOVA terms in the test function).}
  \label{fig:lsqrs=0sigma=0.0}
\end{figure}
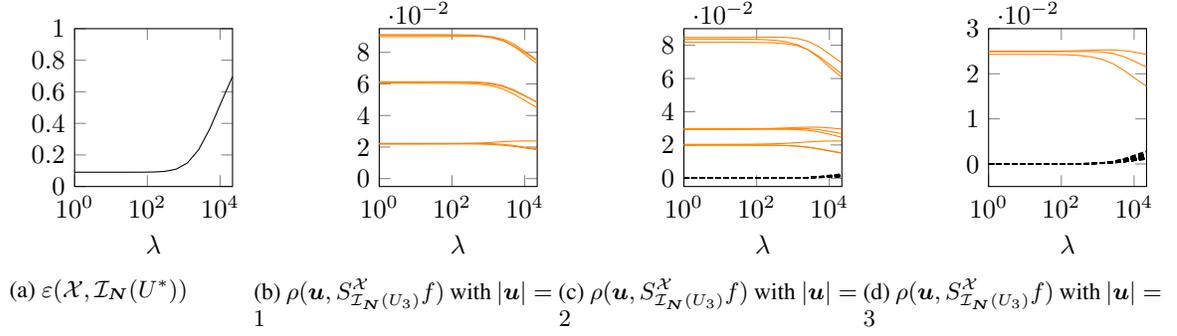
	
	\item
	In the next experiment we consider the effect of adding $10\%$ Gaussian noise which results in \cref{fig:lsqrs=0sigma=0.05}. We observe that the minimal $\L_2$-error $\eps(\mathcal{X},\mathcal{I}_{\b N}(U^\ast)) \approx 0.165$ increases in comparison to the previous experiment as expected. We get a typical behaviour of under- and overfitting in the $\L_2$-error which can be recognized in the global sensitivity indices as well:
	For small $\lambda$ the ANOVA terms $f_{\bm u}$ not occurring in the test function are not zero in our approximation since it fits the noise.
	For large $\lambda$ on the other hand we also penalize ANOVA terms $f_{\bm u}$ which actually occur in the test function as can be seen by the orange lines dropping.
	However, it is evident that we are able to achieve the goal of detecting the active ANOVA terms $f_{\bm u}$, $\u \in U^\ast$, since their sensitivity indices are well-separated from the sensitivity indices of the inactive terms.
	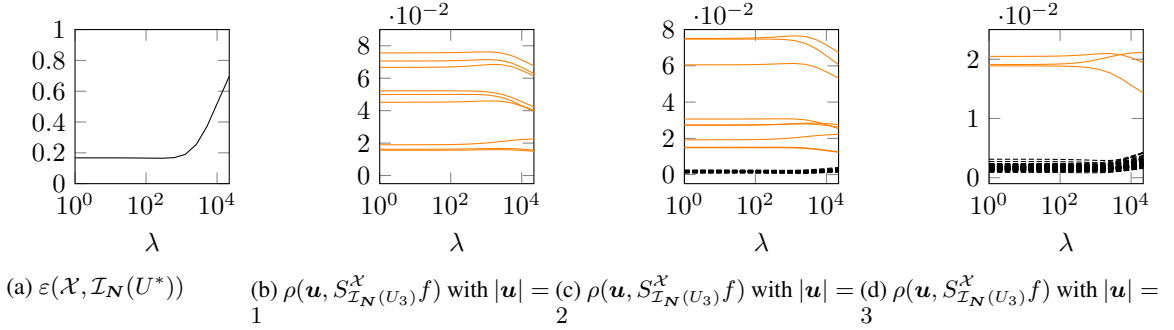
\begin{figure}
  \centering
  \begin{subfigure}[t]{0.24\linewidth}\centering
    \raggedleft
    \begin{tikzpicture}\begin{axis}[
      scale only axis, width = 2.05cm, height = 2.05cm,
      enlarge x limits = 0,
      xmode = log,
      xlabel = $\lambda$,
      ymin = 0, ymax = 1,
    ]
      \addplot[no marks] table[x = lambda, y = L2error] {data/lsqrs=0sigma=0.1.csv};
    \end{axis}\end{tikzpicture}
    \caption{$\eps(\mathcal{X},\mathcal{I}_{\b N}(U^\ast))$}
  \end{subfigure}
  \begin{subfigure}[t]{0.24\linewidth}\centering
    \raggedleft
    \begin{tikzpicture}\begin{axis}[
      scale only axis, width = 2.05cm, height = 2.05cm,
      enlarge x limits = 0,
      xmode = log,
      xlabel = $\lambda$,
      ymin = -0.005, ymax = 0.09,
    ]
      \foreach \u in {1, 2, 3, 4, 5, 6, 7, 8, 9}{
        \addplot[no marks, color = orange] table[x = lambda, y = \u] {data/lsqrs=0sigma=0.1.csv};
      }
    \end{axis}\end{tikzpicture}
    \caption{$\rho(\bm u, S_{\mathcal{I}_{\b N}(U_3)}^{\mathcal X} f)$ with $|\bm u| = 1$}
  \end{subfigure}
  \begin{subfigure}[t]{0.24\linewidth}\centering
    \raggedleft
    \begin{tikzpicture}\begin{axis}[
      scale only axis, width = 2.05cm, height = 2.05cm,
      enlarge x limits = 0,
      xmode = log,
      xlabel = $\lambda$,
      ymin = -0.005, ymax = 0.08,
    ]
      \foreach \u in {12, 14, 15, 16, 17, 19, 23, 24, 27, 28, 29, 34, 35, 36, 37, 39, 45, 46, 48, 57, 58, 59, 67, 68, 69, 78, 89}{
        \addplot[no marks, dash pattern = on 2pt off 1pt] table[x = lambda, y = \u] {data/lsqrs=0sigma=0.1.csv};
      }
      \foreach \u in {13, 18, 25, 26, 38, 47, 49, 56, 79}{
        \addplot[no marks, orange] table[x = lambda, y = \u] {data/lsqrs=0sigma=0.1.csv};
      }
    \end{axis}\end{tikzpicture}
    \caption{$\rho(\bm u, S_{\mathcal{I}_{\b N}(U_3)}^{\mathcal X} f)$ with $|\bm u| = 2$}
  \end{subfigure}
  \begin{subfigure}[t]{0.24\linewidth}\centering
    \raggedleft
    \begin{tikzpicture}\begin{axis}[
      scale only axis, width = 2.05cm, height = 2.05cm,
      enlarge x limits = 0,
      xmode = log,
      xlabel = $\lambda$,
      ymin = -0.001, ymax = 0.025,
    ]
      \foreach \u in {123, 124, 125, 126, 127, 128, 129, 134, 135, 136, 137, 139, 145, 146, 147, 148, 149, 156, 157, 158, 159, 167, 168, 169, 178, 179, 189, 234, 235, 236, 237, 238, 239, 245, 246, 247, 248, 249, 257, 258, 259, 267, 268, 269, 278, 279, 289, 345, 346, 347, 348, 349, 356, 357, 358, 359, 367, 368, 369, 378, 379, 389, 456, 457, 458, 459, 467, 468, 469, 478, 489, 567, 568, 569, 578, 579, 589, 678, 679, 689, 789}{
        \addplot[no marks, dash pattern = on 2pt off 1pt] table[x = lambda, y = \u] {data/lsqrs=0sigma=0.1.csv};
      }
      \foreach \u in {138, 256, 479}{
        \addplot[no marks, orange] table[x = lambda, y = \u] {data/lsqrs=0sigma=0.1.csv};
      }
    \end{axis}\end{tikzpicture}
    \caption{$\rho(\bm u, S_{\mathcal{I}_{\b N}(U_3)}^{\mathcal X} f)$ with $|\bm u| = 3$}
  \end{subfigure}
  \caption{LSQR with small bandwidths given in \eqref{eq:smallbandwidths} on noisy data with $s=0$ (orange: active ANOVA terms $U^\star$, dashed: inactive ANOVA terms in the test function).}
  \label{fig:lsqrs=0sigma=0.05}
\end{figure}
	
	\item
	In this experiment, we are interested in the addition of smoothness information and how it influences the minimization. We choose large bandwidths
	\begin{equation}\label{eq:largebandwidths}
		N_{\u}
		= \begin{cases}
			352 & \text{for } |\bm u| = 1 \\
			20 & \text{for } |\bm u| = 2 \\
			8 & \text{for } |\bm u| = 3,
		\end{cases}
	\end{equation}
	which results in $44\,968$ frequencies and an undetermined optimization problem \eqref{new_min_prob} with $\abs{\mathcal{I}_{\b N}(U_3)} > \abs{\mathcal{X}}$. The addition of the weighted $\ell_2$ regularization term in \eqref{new_min_prob} leads to a reduction of the \textit{search space} to the functions in the Sobolev type space $\sobolev{\omega_s}(\T^d)$ by incorporating the weights $\omega_s$ from \eqref{sobolevweights}. We choose $s = 1.5$ in $\bm{\hat W}(\bm u)$ which propagate functions of this dominating mixed smoothness since for the test function we have $f \in \sobolev{\omega_{1.5-\eps}}(\T^d)$ for every $\eps > 0$. 
	
	The results for the experiment without noise are shown in \cref{fig:lsqrs=32sigma=0.0}. We see that in this undetermined setting, the incorporation of smoothness information allowed us to efficiently solve the problem and with $\eps(\mathcal{X},\mathcal{I}_{\b N}(U^\ast)) \approx 1.8 \cdot 10^{-2}$ even achieve a better $\L_2$-error than in the previous experiments. The active ANOVA terms in $U^\ast$ are also clearly separable from the inactive ANOVA terms. We also recognize the different smoothness levels of the ANOVA terms $f_{\bm u}$, $\u \in U^\ast$, through the sensitivity indices as described in (i).
	\begin{figure}
  \centering
  \begin{subfigure}[t]{0.24\linewidth}\centering
    \raggedleft
    \begin{tikzpicture}\begin{axis}[
      scale only axis, width = 2.1cm, height = 2.1cm,
      enlarge x limits = 0,
      xmode = log,
      xlabel = $\lambda$,
      ymin = 0, ymax = 0.3,
    ]
      \addplot[no marks] table[x = lambda, y = L2error] {data/lsqrs=1.5sigma=0.0.csv};
    \end{axis}\end{tikzpicture}
    \caption{$\eps(\mathcal{X},\mathcal{I}_{\b N}(U^\ast))$}
  \end{subfigure}
  \begin{subfigure}[t]{0.24\linewidth}\centering
    \raggedleft
    \begin{tikzpicture}\begin{axis}[
      scale only axis, width = 2.1cm, height = 2.1cm,
      enlarge x limits = 0,
      xmode = log,
      xlabel = $\lambda$,
      ytick = {0,0.1,0.2},
      ymin = -0.005, ymax = 0.2,
    ]
      \foreach \u in {1, 2, 3, 4, 5, 6, 7, 8, 9}{
        \addplot[no marks, color = orange] table[x = lambda, y = \u] {data/lsqrs=1.5sigma=0.0.csv};
      }
    \end{axis}\end{tikzpicture}
    \caption{$\rho(\bm u, S_{\mathcal{I}_{\b N}(U_3)}^{\mathcal X} f)$ with $|\bm u| = 1$}
  \end{subfigure}
  \begin{subfigure}[t]{0.24\linewidth}\centering
    \raggedleft
    \begin{tikzpicture}\begin{axis}[
      scale only axis, width = 2.1cm, height = 2.1cm,
      enlarge x limits = 0,
      xmode = log,
      xlabel = $\lambda$,
      ymin = -0.005, ymax = 0.09,
    ]
      \foreach \u in {12, 14, 15, 16, 17, 19, 23, 24, 27, 28, 29, 34, 35, 36, 37, 39, 45, 46, 48, 57, 58, 59, 67, 68, 69, 78, 89}{
        \addplot[no marks, dash pattern = on 2pt off 1pt] table[x = lambda, y = \u] {data/lsqrs=1.5sigma=0.0.csv};
      }
      \foreach \u in {13, 18, 25, 26, 38, 47, 49, 56, 79}{
        \addplot[no marks, orange] table[x = lambda, y = \u] {data/lsqrs=1.5sigma=0.0.csv};
      }
    \end{axis}\end{tikzpicture}
    \caption{$\rho(\bm u, S_{\mathcal{I}_{\b N}(U_3)}^{\mathcal X} f)$ with $|\bm u| = 2$}
  \end{subfigure}
  \begin{subfigure}[t]{0.24\linewidth}\centering
    \raggedleft
    \begin{tikzpicture}\begin{axis}[
      scale only axis, width = 2.1cm, height = 2.1cm,
      enlarge x limits = 0,
      xmode = log,
      xlabel = $\lambda$,
      ytick = {0.01,0.02,0},
      ymin = -0.001, ymax = 0.02,
    ]
      \foreach \u in {123, 124, 125, 126, 127, 128, 129, 134, 135, 136, 137, 139, 145, 146, 147, 148, 149, 156, 157, 158, 159, 167, 168, 169, 178, 179, 189, 234, 235, 236, 237, 238, 239, 245, 246, 247, 248, 249, 257, 258, 259, 267, 268, 269, 278, 279, 289, 345, 346, 347, 348, 349, 356, 357, 358, 359, 367, 368, 369, 378, 379, 389, 456, 457, 458, 459, 467, 468, 469, 478, 489, 567, 568, 569, 578, 579, 589, 678, 679, 689, 789}{
        \addplot[no marks, dash pattern = on 2pt off 1pt] table[x = lambda, y = \u] {data/lsqrs=1.5sigma=0.0.csv};
      }
      \foreach \u in {138, 256, 479}{
        \addplot[no marks, orange] table[x = lambda, y = \u] {data/lsqrs=1.5sigma=0.0.csv};
      }
    \end{axis}\end{tikzpicture}
    \caption{$\rho(\bm u, S_{\mathcal{I}_{\b N}(U_3)}^{\mathcal X} f)$ with $|\bm u| = 3$}
  \end{subfigure}
  \caption{LSQR with large bandwidths given in \eqref{eq:largebandwidths} on exact data with $s=1.5$ (orange: active ANOVA terms $U^\star$, dashed: inactive ANOVA terms in the test function).}
  \label{fig:lsqrs=32sigma=0.0}
\end{figure}
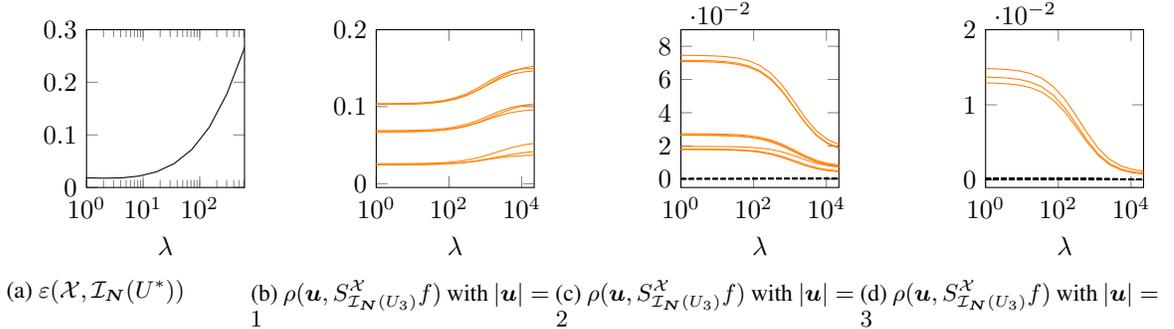
	
	\item
	The same experiment with the addition of $10\%$ Gaussian noise can be seen in \cref{fig:lsqrs=32sigma=0.1}.
	In comparison to the noisy experiment with small bandwidths in (ii) we achieve a smaller minimal $\L_2$-error of $\eps(\mathcal{X},\mathcal I(U^\ast)) \approx 0.189$.
	In addition, the different smoothness levels of the ANOVA terms $f_{\bm u}$, $\u \in U^\ast$, are clearly better distinguishable than in experiment (ii).
	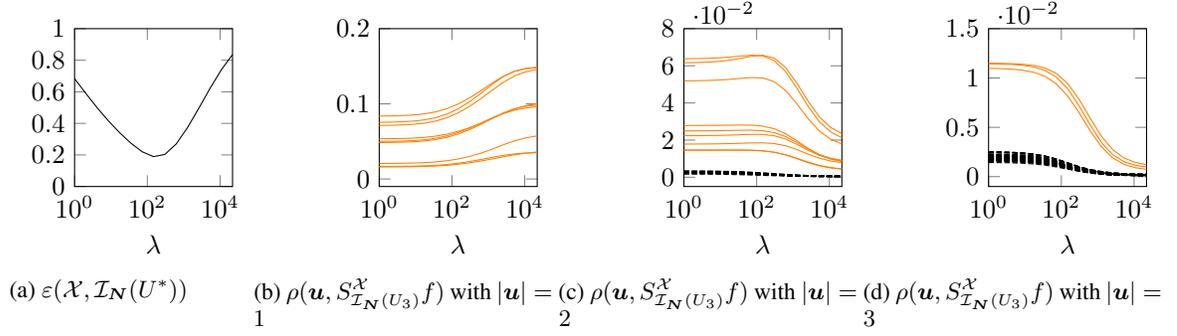
\begin{figure}
  \centering
  \begin{subfigure}[t]{0.24\linewidth}\centering
    \raggedleft
    \begin{tikzpicture}\begin{axis}[
      scale only axis, width = 2.1cm, height = 2.1cm,
      enlarge x limits = 0,
      xmode = log,
      xlabel = $\lambda$,
      ymin = 0, ymax = 1,
    ]
      \addplot[no marks] table[x = lambda, y = L2error] {data/lsqrs=1.5sigma=0.1.csv};
    \end{axis}\end{tikzpicture}
    \caption{$\eps(\mathcal{X},\mathcal{I}_{\b N}(U^\ast))$}
  \end{subfigure}
  \begin{subfigure}[t]{0.24\linewidth}\centering
    \raggedleft
    \begin{tikzpicture}\begin{axis}[
      scale only axis, width = 2.1cm, height = 2.1cm,
      enlarge x limits = 0,
      xmode = log,
      xlabel = $\lambda$,
      ytick = {0, 0.1,0.2},
      ymin = -0.01, ymax = 0.2,
    ]
      \foreach \u in {1, 2, 3, 4, 5, 6, 7, 8, 9}{
        \addplot[no marks, color = orange] table[x = lambda, y = \u] {data/lsqrs=1.5sigma=0.1.csv};
      }
    \end{axis}\end{tikzpicture}
    \caption{$\rho(\bm u, S_{\mathcal{I}_{\b N}(U_3)}^{\mathcal X} f)$ with $|\bm u| = 1$}
  \end{subfigure}
  \begin{subfigure}[t]{0.24\linewidth}\centering
    \raggedleft
    \begin{tikzpicture}\begin{axis}[
      scale only axis, width = 2.1cm, height = 2.1cm,
      enlarge x limits = 0,
      xmode = log,
      xlabel = $\lambda$,
      ymin = -0.005, ymax = 0.08,
    ]
      \foreach \u in {12, 14, 15, 16, 17, 19, 23, 24, 27, 28, 29, 34, 35, 36, 37, 39, 45, 46, 48, 57, 58, 59, 67, 68, 69, 78, 89}{
        \addplot[no marks, dash pattern = on 2pt off 1pt] table[x = lambda, y = \u] {data/lsqrs=1.5sigma=0.1.csv};
      }
      \foreach \u in {13, 18, 25, 26, 38, 47, 49, 56, 79}{
        \addplot[no marks, orange] table[x = lambda, y = \u] {data/lsqrs=1.5sigma=0.1.csv};
      }
    \end{axis}\end{tikzpicture}
    \caption{$\rho(\bm u, S_{\mathcal{I}_{\b N}(U_3)}^{\mathcal X} f)$ with $|\bm u| = 2$}
  \end{subfigure}
  \begin{subfigure}[t]{0.24\linewidth}\centering
    \raggedleft
    \begin{tikzpicture}\begin{axis}[
      scale only axis, width = 2.1cm, height = 2.1cm,
      enlarge x limits = 0,
      xmode = log,
      xlabel = $\lambda$,
      ymin = -0.001, ymax = 0.015,
    ]
      \foreach \u in {123, 124, 125, 126, 127, 128, 129, 134, 135, 136, 137, 139, 145, 146, 147, 148, 149, 156, 157, 158, 159, 167, 168, 169, 178, 179, 189, 234, 235, 236, 237, 238, 239, 245, 246, 247, 248, 249, 257, 258, 259, 267, 268, 269, 278, 279, 289, 345, 346, 347, 348, 349, 356, 357, 358, 359, 367, 368, 369, 378, 379, 389, 456, 457, 458, 459, 467, 468, 469, 478, 489, 567, 568, 569, 578, 579, 589, 678, 679, 689, 789}{
        \addplot[no marks, dash pattern = on 2pt off 1pt] table[x = lambda, y = \u] {data/lsqrs=1.5sigma=0.1.csv};
      }
      \foreach \u in {138, 256, 479}{
        \addplot[no marks, orange] table[x = lambda, y = \u] {data/lsqrs=1.5sigma=0.1.csv};
      }
    \end{axis}\end{tikzpicture}
    \caption{$\rho(\bm u, S_{\mathcal{I}_{\b N}(U_3)}^{\mathcal X} f)$ with $|\bm u| = 3$}
  \end{subfigure}
  \caption{LSQR with large bandwidths given in \eqref{eq:largebandwidths} on noisy data with $s=1.5$ (orange: active ANOVA terms $U^\star$, dashed: inactive ANOVA terms in the test function).}
  \label{fig:lsqrs=32sigma=0.1}
\end{figure}
	
\end{enumerate}

\begin{remark}
	The optimal choice of $\lambda$ in this paper is determined via the minimum $\L_2$-error $\eps(\mathcal{X},\mathcal{I}_{\b N}(U))$ from \eqref{L2error}. As mentioned before, this error can only be explicitly calculated if we generate the data as evaluations from a known synthetic test function. In the general case, one could use cross-validation techniques, e.g., with the methods in \cite{BaHiPo19} which allow for fast computation of the cross-validation score. This score can then be used to choose a $\lambda$ which will be close to the $\L_2(\mathbb T^d)$-optimal one, i.e., avoiding over- or underfitting.
\end{remark}

\subsection{Group Lasso}\label{sec:fista} %%%%%%%%%%%%%%%%%%%%%%%%%%%%%%%%%%%%%%%%%%%%%%%%%%%%%%%%%

As it is promoted in \cref{sec:anova}, after choosing a superposition threshold $d_s$ and the set of ANOVA terms $U_{d_s}$, we often have only a few of those ANOVA terms active. In other words, there is sparsity in the groups $\bm u \in U_{d_s}$.
The ANOVA terms $f_{\bm u}$ corresponding to these active groups, however, do not inherit this sparse structure, i.e., there is no sparsity in their Fourier coefficients.
Therefore, a method is needed which promotes sparsity on the scale of the groups $\bm u\in U_{d_s}$, but not within the coefficients of this group $\bm{\hat f}(\bm u)$, cf.\ \eqref{eq:somafm}. Incorporating this type of sparsity in the regularization combines the previously required sensitivity analysis into an automated process during the minimization.

A method which accomplishes this mixture of sparsity and non-sparsity constraints is called group lasso and was introduced in \cite{YL06}.
With group lasso, one seeks for the solution $\bm{\hat f}^\star$ of
\begin{equation}\label{eq:alligatoah}
	\min_{\bm{\hat f}} \frac 12 \|\bm y-\bm F(\mathcal X,\mathcal{I}_{\b N}(U))\bm{\hat f}\|_2^2
	+\lambda \sum_{\bm u\in U} \|\bm{\hat f}(\bm u)\|_{\bm{\hat W}(\bm u)}
\end{equation}
with grouped index set $\mathcal{I}_{\b N}(U)$, $\b N = (N_{\u})_{\u \in U}$, see \eqref{groupedset}, where we use the weighted norm
$\|\bm{\hat f}(\bm u)\|_{\bm{\hat W}(\bm u)}^2 = \bm{\hat f}(\bm u)\herm \bm{\hat W}(\bm u) \bm{\hat f}(\bm u)$ and $\bm{\hat W}(\bm u) = \diag(\omega(\bm k))_{\bm k\in\mathcal I_{N_{\bm u}}^{\bm u, d}}$ diagonal matrices with weight functions $\omega : \mathbb Z^d \to [1,\infty)$ for $\bm u \in U$.
Notice the lack of the square in the regularization term.
By incorporating squares in \eqref{eq:alligatoah} we would attain the same functional as in \cref{sec:lsqr}.
However, omitting them allows for an $\ell_1$-norm structure around the groups for a sparsifying effect similar to the basic lasso approach itself.
In this group case, it pushes towards a small number of active ANOVA terms.

\begin{remark}
	With differently sized groups, i.e., a different amount of Fourier coefficients in each group $\mathcal I_{N_{\u}}^{\bm u, d}$ of the index set, see \eqref{metallica}, lasso would set smaller groups with fewer frequencies to zero for sparsity as data fitting is easier with groups of more degrees of freedom. This could be tackled by introducing an additional weight for $\bm u\in U_{d_s}$ counterbalancing the group size. We work with groups of the same amount of frequencies throughout this section making this additional weight obsolete.
\end{remark}

The missing squares bring another difficulty with them, namely, the lack of differentiability in the regularization term.
Thus \eqref{eq:alligatoah} is a non-smooth convex minimization problem, which we tackle by the use of a proximal gradient method.
Entranced by the computational simplicity and quadratic rate of convergence, we use the fast iterative shrinking-thresholding algorithm (FISTA), cf.~\cite{BT09}, as our algorithm of choice.

In order to formulate the algorithm, we have to introduce the projection $p_L$.
For a vector $\bm{\hat h}$ of Fourier coefficients, $p_L(\bm{\hat h})$ is given by the minimizer of
\begin{equation}\label{eq:projection}
	%  p_L(\bm{\hat h})
	%  = \argmin_{\bm{\hat f}}\left\{
	\frac 12\left\|\bm{\hat f}-\left(\bm{\hat h}-\frac 1L\bm F\herm(\mathcal X,\mathcal{I}_{\b N}(U))\left(\bm F(\mathcal X,\mathcal{I}_{\b N}(U))\bm{\hat h}-\bm y\right)\right)\right\|_2^2
	+ \frac{\lambda}{L} \sum_{\bm u\in U}\|\bm{\hat f}_{\bm u}\|_{\bm{\hat W}(\bm u)} 
	%  \right\}. 
	.
\end{equation}
With that, we are able to write the FISTA algorithm adopted to our setting:
\begin{algorithm}[ht]
	\medskip
	
	\textbf{Input: }initial guess for the Fourier coefficents $\bm{\hat f}_0$,\\
	\phantom{\textbf{Input: }}stepsize parameters $L_0 > 0$, $\eta > 1$ and\\
	\phantom{\textbf{Input: }}maximal iteration count $K\in\mathbb N$
	\medskip
	
	\textbf{Output:} $\bm{\hat f}^\star = \bm{\hat f}^{(K)}$ minimizer of \eqref{eq:alligatoah}
	\medskip
	
	\begin{algorithmic}[1]
		\STATE{
			$\bm{\hat h}^{(1)} \leftarrow \bm{\hat f}^{(0)}$, $t_1 \leftarrow 1$
		}
		\FOR{$k=1, \dots, K$}
		\STATE{$L_k \leftarrow L_{k-1}$}
		\STATE{
			$ \bm{\hat f}^{(k)} \leftarrow p_{L_k}(\bm{\hat h}^{(k)}) $
		}
		\WHILE{
			$
			\|\bm y-\bm F(\mathcal X,\mathcal{I}_{\b N}(U)) \bm{\hat f}^{(k)}\|_2^2
			- \|\bm y-\bm F(\mathcal X,\mathcal{I}_{\b N}(U)) \bm{\hat h}^{(k)}\|_2^2
			$ \\
			\hfill$
			>2\langle \bm{\hat f}^{(k)}-\bm{\hat h}^{(k)},\, \bm F\herm(\mathcal X,\mathcal{I}_{\b N}(U))(\bm F(\mathcal X,\mathcal{I}_{\b N}(U))\bm{\hat h}^{(k)}-\bm y) \rangle
			+L_k \| \bm{\hat f}^{(k)}-\bm{\hat h}^{(k)} \|_2^2
			$
		}
		\STATE{$L_k \leftarrow \eta L_k$}
		\STATE{
			$ \bm{\hat f}^{(k)} \leftarrow p_{L_k}(\bm{\hat h}^{(k)}) $
		}
		\ENDWHILE
		\STATE{
			$ t_{k+1} \leftarrow (1+\sqrt{1+4t_k^2})/2 $
		}
		\STATE{
			$ \bm{\hat h}^{({k+1})} \leftarrow \bm{\hat f}^{(k)}+(t_k-1)/(t_{k+1})(\bm{\hat f}^{(k)}-\bm{\hat f}^{({k-1})}) $
		}
		\ENDFOR
	\end{algorithmic}
	\caption{FISTA for group lasso}
	\label{algo:fista}
\end{algorithm}

\begin{remark}
	The computational most expensive part is the projection $p_L$ which we will reduce to one multiplication with $\bm F(\mathcal X,\mathcal{I}_{\b N}(U))$ and one with its adjoint which we realize by the grouped transformation from \cref{sec:gft}.
	Note that there is also a variant of the FISTA algorithm which uses constant step size $L_k$ which makes the inner loop obsolete.
	For that, the complexity of the FISTA step boils down to the complexity of computing $p_L$.
	The quadratic convergence rate assures us, that a small number of iterations is sufficient.
\end{remark}

To compute the projection $p_L$ we develop an explicit formula which will be done in the next theorem with the use of the following lemma.

\begin{lemma}\label{lemma:lennykrawitz}
	Let $\mathcal I \subset \Z^d$ be a finite frequency index set, $\omega:\mathbb Z^d\to[0,\infty)$ a weight function, and $\bm{\hat W} = \diag(\omega(\k))_{\k\in\mathcal I}$.
	Further, let $\bm x_1^\star(\lambda)$ and $\bm x_2^\star(\xi)$ be the minimizers of
	$$
	\frac 12 \|\bm x - \bm y\|_2^2+\lambda\|\bm x\|_{\bm{\hat W}}
	\quad\text{and}\quad
	\frac 12 \|\bm x - \bm y\|_2^2+\xi\|\bm x\|_{\bm{\hat W}}^2,
	$$
	respectively.
	Then for $\xi > 0$ we have
	\begin{enumerate}[(i)]
		\item
		$$
		\bm x_2^\star(\xi)
		= \diag\left(\frac{1}{1+2\xi\omega(\k)}\right)_{\k\in\mathcal I}\bm y,
		$$
		\item
		$$
		\bm x_1^\star(\lambda) 
		= 
		\bm x_2^\star(\xi)
		\quad\text{for}\quad
		\lambda
		= \left\|\left(\frac{y_{\k}}{1/(2\xi)+\omega(\k)}\right)_{\k\in\mathcal I}\right\|_{\bm{\hat W}},
		$$
		\item and
		$$
		\bm x_1^\star(\lambda) = \bm 0
		\quad\text{for}\quad
		\lambda 
		\ge \left\|\left(\frac{y_{\k}}{\omega(\k)}\right)_{\k\in\mathcal I}\right\|_{\bm{\hat W}}.
		$$
	\end{enumerate}
\end{lemma}

\begin{proof}
	\begin{enumerate}[(i)]
		\item 
		We simply calculate the roots of the gradient
		$$
		\bm x-\bm y + 2\xi\bm{\hat W}\bm x
		\overset != \bm 0
		\quad\Leftrightarrow\quad
		\bm x_2^\star(\xi)
		= \diag\left(\frac{1}{1+2\xi\omega(\k)}\right)_{\k\in\mathcal I}\bm y.
		$$
		The convexity of the function assures the minimizing property.
		\item
		Similarly to (i), computing $\bm x_1^\star(\lambda)$ can be done by root-finding of
		$$
		\nabla_{\bm x} \left( \frac 12\|\bm x-\bm y\|_2^2+\lambda\|\bm x\|_{\bm{\hat W}} \right)
		= \bm x-\bm y + \frac{\lambda}{\|\bm x\|_{\bm{\hat W}}} \diag(\omega(\k))_{\k\in\mathcal I}\bm x
		$$
		since the function in question is strictly convex.
		Using the ansatz $\bm x = \bm x_2^\star(\xi)$ and (i) leads to
		$$
		\left(\diag\left(\frac{1}{1+2\xi\omega(\k)}\right)_{\k\in\mathcal I} - \bm I
		+ \frac{\lambda
			\diag\left(\frac{\omega(\k)}{1+2\xi\omega(\k)}\right)_{\k\in\mathcal I}
		}{\left\|
			\diag\left(\frac{1}{1+2\xi\omega(\k)}\right)_{\k\in\mathcal I}\bm y
			\right\|_{\bm{\hat W}}}
		\right)\bm y
		\overset != \bm0.
		$$
		This is certainly fulfilled if the diagonal matrix is zero in each component, i.e.,
		$$
		\frac{1}{1+2\xi\omega(\k)}-1+\frac{\lambda}{\left\|\diag\left(\frac{1}{1+2\xi\omega(\k)}\right)_{\k\in\mathcal I}\bm y\right\|_{\bm{\hat W}}}\frac{\omega(\k)}{1+2\xi\omega(\k)}
		\overset != 0,
		$$
		or equivalently,
		$$
		\lambda
		\overset != \left\|\diag\left(\frac{1}{1/(2\xi)+\omega(\k)}\right)_{\k\in\mathcal I}\bm y\right\|_{\bm{\hat W}}.
		$$
		\item
		We want to minimize
		$$
		\frac 12\|\bm x-\bm y\|_2^2+\lambda\|\bm x\|_{\bm{\hat W}}.
		$$
		Using polar coordinates we see that varying the arguments in $\bm x$ does not change the second summand.
		Hence, $\bm x$ and $\bm y$ have to have the same arguments in each component.
		Without loss of generality, we restrict to the moduli of these numbers, i.e., positive numbers.
		
		Then
		\begin{multline*}
			\frac 12\|\bm 0 - \bm y\|_2^2
			+\lambda \|\bm 0\|_{\bm{\hat W}} 
			=
			\frac 12 \|-\bm x+\bm x- \bm y\|_2^2 \\
			\le \frac 12\|\bm x- \bm y\|_2^2 + 
			\left\langle\diag\left(\sqrt{\omega(\k)}\right)_{\k\in\mathcal I}\bm x,\,
			\diag\left(\frac{1}{\sqrt{\omega(\k)}}\right)_{\k\in\mathcal I}\bm y\right\rangle.
		\end{multline*}
		By the Cauchy-Schwarz inequality,  we obtain
		$$
		\frac 12\|\bm 0 - \bm y\|_2^2
		+\lambda \|\bm 0\|_2 
		\le \frac 12\|\bm x- \bm y\|_2^2 + \|\bm x\|_{\bm{\hat W}}\left\|\diag\left(1/\sqrt{\omega(\k)}\right)_{\k\in\mathcal I}\bm y\right\|_2.
		$$
		Now making use of $\|\diag(1/\sqrt{\omega(\k)})_{\k\in\mathcal I} \bm y\|_2 < \lambda$ we end up with
		$$
		\frac 12\|\bm 0 - \bm y\|_2^2
		+\lambda \|\bm 0\|_2 
		\le \frac 12\|\bm x- \bm y\|_2^2 + \lambda\|\bm x\|_{\bm{\hat W}}
		$$
		which holds for all $\bm x$ and, hence, $\bm 0$ is the minimizer for this case.
	\end{enumerate}
\end{proof}

Now, using \cref{lemma:lennykrawitz}, we prove an explicit formula for the projection $p_L$ defined via the minimizer of \eqref{eq:projection}.

\begin{theorem}\label{theorem:projection}
	Let $U$ be a subset of $\mathcal P(\{1,\dots,d\})$ and diagonal matrices $\bm{\hat W}(\bm u) = \diag(\omega(\k))_{\k\in\mathcal I_{N_{\bm u}}^{\bm u, d}}$ with $ I_{N_{\bm u}}^{\bm u, d}$ from \eqref{metallica} and strictly positive weights for $\bm u\in U$.
	For any $\bm y\in\mathbb C^N$,
	$N = \sum_{\bm u\in U} (1-N_{\u})^{\au}$, and $\lambda > 0$, the minimizer of
	$$
	\frac 12 \|\bm x-\bm y\|_2^2+\lambda\sum_{\bm u\in U} \|\bm x_{\bm u}\|_{\bm{\hat W(\bm u)}}
	$$
	is given by
	$$
	\bm x^\star_{\bm u}
	= \begin{cases}
		\diag\left(\frac{1}{1+\xi(\bm u)\omega(\k)}\right)_{\k\in\mathcal I_{N_{\u}}^{\bm u, d}}\bm y & \text{for } \lambda \le \left\|\left(\frac{y_{\k}}{\omega(\k)}\right)_{\k\in\mathcal I_{N_{\u}}^{\bm u, d}}\right\|_{\bm{\hat W}(\bm u)} \\
		\bm 0 & \text{otherwise}
	\end{cases}
	$$
	for $\bm u\in U$ with $\xi(\bm u)$ fulfilling
	$$
	\left\|\left(\frac{y_{\k}}{1/\xi(\bm u)+\omega(\k)}\right)_{\k\in\mathcal I_{N_{\u}}^{\bm u, d}}\right\|_{\bm{\hat W}(\bm u)}
	= \lambda.
	$$
\end{theorem}

\begin{proof}
	Because of the structure of the objective function, we can minimize for every $\bm u\in U$ separately.
	\cref{lemma:lennykrawitz} (i)-(iii) then results the assertion.
\end{proof}

\begin{remark}
	In \cref{theorem:projection} we still have to find $\xi(\bm u)$ such that
	$$
	t(\xi)
	\coloneqq \left\|\left(\frac{y_{\k}}{1/\xi(\bm u)+\omega(\k)}\right)_{\k\in\mathcal I_{N_{\u}}^{\bm u, d}}\right\|_{\bm{\hat W}(\bm u)}^2
	= \lambda^2.
	$$
	We do that in two steps:
	\begin{enumerate}[(i)]
		\item
		First, we determine whether $t(1) \ge \lambda^2$ or not.
		If so, $\xi$ has to be smaller than one because of the monotonicity of the objective function.
		Then we use the bisection method on $t(\cdot)$ with the initial interval $[0, 1]$.
		Otherwise we do the same for the function $t(1/\cdot)$.
		\item
		To obtain a more precise result we use some iterations with Newton's method afterwards.
	\end{enumerate}
\end{remark}

Now for the numerical experiments. As in \cref{sec:lsqr}, we use the nine-dimensional function \eqref{eq:testfun} which we sample at $10\,000$ uniform i.i.d.\,random nodes $\mathcal X \subset \mathbb T^9$.
For four different settings we launched \cref{algo:fista} with $\lambda\in [\e^0,\e^{10}]$.
We started with the initial guess $\bm{\hat f}_0 = \bm 0$ and the biggest $\lambda = \e^{10}$ with the minimizer $\bm{\hat f}^\star(\lambda)$ which is presumably close to $\bm 0$ for this case.
For the remaining values of $\lambda$ we used the last solution as new initial guess since $\bm{\hat f}^\star(\lambda)$ depends continuously on $\lambda$.
For evaluation, we plotted the $\L_2$-error and the global sensitivity indices~\eqref{eq:gsi} for each $|\bm u|=1,2,3$ separately. Note that the errors and global sensitivity indices have been averaged over 100 random draws of the node set $\mathcal X$. 

\begin{enumerate}[(i)]
	\item
	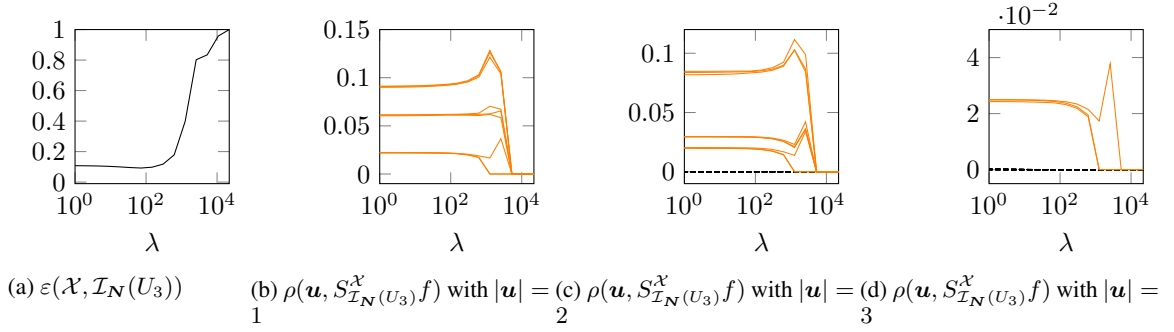
\begin{figure}
  \centering
  \begin{subfigure}[t]{0.24\linewidth}\centering
    \raggedleft
    \begin{tikzpicture}\begin{axis}[
      scale only axis, width = 2.05cm, height = 2.05cm,
      enlarge x limits = 0,
      xmode = log,
      xlabel = $\lambda$,
      ymin = -0.01, ymax = 1,
    ]
      \addplot[no marks] table[x = lambda, y = L2error] {data/fistas=0sigma=0.csv};
    \end{axis}\end{tikzpicture}
    \caption{$\eps(\mathcal{X},\mathcal{I}_{\b N}(U_3))$}
  \end{subfigure}
  \begin{subfigure}[t]{0.24\linewidth}\centering
    \raggedleft
    \begin{tikzpicture}\begin{axis}[
    		yticklabel style={
    			/pgf/number format/fixed,
    			/pgf/number format/precision=5
    		},
    		scaled y ticks=false,
      scale only axis, width = 2.05cm, height = 2.05cm,
      enlarge x limits = 0,
      xmode = log,
      xlabel = $\lambda$,
      ymin = -0.01, ymax = 0.15,
    ]
      \foreach \u in {1, 2, 3, 4, 5, 6, 7, 8, 9}{
        \addplot[no marks, color = orange] table[x = lambda, y = \u] {data/fistas=0sigma=0.csv};
      }
    \end{axis}\end{tikzpicture}
    \caption{$\rho(\bm u, S_{\mathcal{I}_{\b N}(U_3)}^{\mathcal X} f)$ with $|\bm u| = 1$}
  \end{subfigure}
  \begin{subfigure}[t]{0.24\linewidth}\centering
    \raggedleft
    \begin{tikzpicture}\begin{axis}[
    		yticklabel style={
    			/pgf/number format/fixed,
    			/pgf/number format/precision=5
    		},
    		scaled y ticks=false,
      scale only axis,width = 2.05cm, height = 2.05cm,
      enlarge x limits = 0,
      xmode = log,
      xlabel = $\lambda$,
      ymin = -0.01, ymax = 0.12,
    ]
      \foreach \u in {12, 14, 15, 16, 17, 19, 23, 24, 27, 28, 29, 34, 35, 36, 37, 39, 45, 46, 48, 57, 58, 59, 67, 68, 69, 78, 89}{
        \addplot[no marks, dash pattern = on 2pt off 1pt] table[x = lambda, y = \u] {data/fistas=0sigma=0.csv};
      }
      \foreach \u in {13, 18, 25, 26, 38, 47, 49, 56, 79}{
        \addplot[no marks, orange] table[x = lambda, y = \u] {data/fistas=0sigma=0.csv};
      }
    \end{axis}\end{tikzpicture}
    \caption{$\rho(\bm u, S_{\mathcal{I}_{\b N}(U_3)}^{\mathcal X} f)$ with $|\bm u| = 2$}
  \end{subfigure}
  \begin{subfigure}[t]{0.24\linewidth}\centering
    \raggedleft
    \begin{tikzpicture}\begin{axis}[
      scale only axis, width = 2.05cm, height = 2.05cm,
      enlarge x limits = 0,
      xmode = log,
      xlabel = $\lambda$,
      ymin = -0.005, ymax = 0.05,
    ]
      \foreach \u in {123, 124, 125, 126, 127, 128, 129, 134, 135, 136, 137, 139, 145, 146, 147, 148, 149, 156, 157, 158, 159, 167, 168, 169, 178, 179, 189, 234, 235, 236, 237, 238, 239, 245, 246, 247, 248, 249, 257, 258, 259, 267, 268, 269, 278, 279, 289, 345, 346, 347, 348, 349, 356, 357, 358, 359, 367, 368, 369, 378, 379, 389, 456, 457, 458, 459, 467, 468, 469, 478, 489, 567, 568, 569, 578, 579, 589, 678, 679, 689, 789}{
        \addplot[no marks, dash pattern = on 2pt off 1pt] table[x = lambda, y = \u] {data/fistas=0sigma=0.csv};
      }
      \foreach \u in {138, 256, 479}{
        \addplot[no marks, orange] table[x = lambda, y = \u] {data/fistas=0sigma=0.csv};
      }
    \end{axis}\end{tikzpicture}
    \caption{$\rho(\bm u, S_{\mathcal{I}_{\b N}(U_3)}^{\mathcal X} f)$ with $|\bm u| = 3$}
  \end{subfigure}
  \caption{FISTA with small bandwidths given in \eqref{eq:smallbandwidths} on exact data with $s=0$ (orange: active ANOVA terms $U^\star$, dashed: inactive ANOVA terms in the test function).}
  \label{fig:fistas=0sigma=0.0}
\end{figure}
	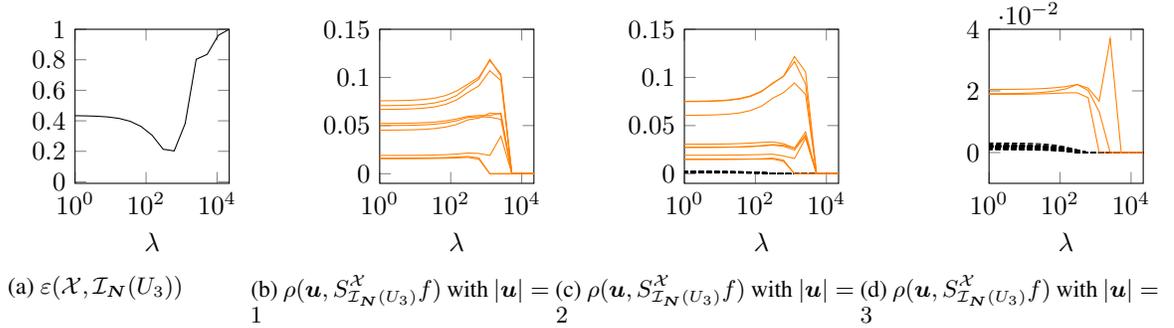
\begin{figure}
  \centering
  \begin{subfigure}[t]{0.24\linewidth}\centering
    \raggedleft
    \begin{tikzpicture}\begin{axis}[
      scale only axis, width = 2.05cm, height = 2.05cm,
      enlarge x limits = 0,
      xmode = log,
      xlabel = $\lambda$,
      ymin = -0.01, ymax = 1,
    ]
      \addplot[no marks] table[x = lambda, y = L2error] {data/fistas=0.0sigma=0.1.csv};
    \end{axis}\end{tikzpicture}
    \caption{$\eps(\mathcal{X},\mathcal{I}_{\b N}(U_3))$}
  \end{subfigure}
  \begin{subfigure}[t]{0.24\linewidth}\centering
    \raggedleft
    \begin{tikzpicture}\begin{axis}[
		 yticklabel style={
		 	/pgf/number format/fixed,
		 	/pgf/number format/precision=5
		 },
		 scaled y ticks=false,
      scale only axis, width = 2.05cm, height = 2.05cm,
      enlarge x limits = 0,
      xmode = log,
      xlabel = $\lambda$,
      ymin = -0.01, ymax = 0.15,
    ]
      \foreach \u in {1, 2, 3, 4, 5, 6, 7, 8, 9}{
        \addplot[no marks, color = orange] table[x = lambda, y = \u] {data/fistas=0.0sigma=0.1.csv};
      }
    \end{axis}\end{tikzpicture}
    \caption{$\rho(\bm u, S_{\mathcal{I}_{\b N}(U_3)}^{\mathcal X} f)$ with $|\bm u| = 1$}
  \end{subfigure}
  \begin{subfigure}[t]{0.24\linewidth}\centering
    \raggedleft
    \begin{tikzpicture}\begin{axis}[
    		yticklabel style={
    			/pgf/number format/fixed,
    			/pgf/number format/precision=5
    		},
    		scaled y ticks=false,
      scale only axis, width = 2.05cm, height = 2.05cm,
      enlarge x limits = 0,
      xmode = log,
      xlabel = $\lambda$,
      ymin = -0.01, ymax = 0.15,
    ]
      \foreach \u in {12, 14, 15, 16, 17, 19, 23, 24, 27, 28, 29, 34, 35, 36, 37, 39, 45, 46, 48, 57, 58, 59, 67, 68, 69, 78, 89}{
        \addplot[no marks, dash pattern = on 2pt off 1pt] table[x = lambda, y = \u] {data/fistas=0.0sigma=0.1.csv};
      }
      \foreach \u in {13, 18, 25, 26, 38, 47, 49, 56, 79}{
        \addplot[no marks, orange] table[x = lambda, y = \u] {data/fistas=0.0sigma=0.1.csv};
      }
    \end{axis}\end{tikzpicture}
    \caption{$\rho(\bm u, S_{\mathcal{I}_{\b N}(U_3)}^{\mathcal X} f)$ with $|\bm u| = 2$}
  \end{subfigure}
  \begin{subfigure}[t]{0.24\linewidth}\centering
    \raggedleft
    \begin{tikzpicture}\begin{axis}[
      scale only axis,width = 2.05cm, height = 2.05cm,
      enlarge x limits = 0,
      xmode = log,
      xlabel = $\lambda$,
      ymin = -0.01, ymax = 0.04,
    ]
      \foreach \u in {123, 124, 125, 126, 127, 128, 129, 134, 135, 136, 137, 139, 145, 146, 147, 148, 149, 156, 157, 158, 159, 167, 168, 169, 178, 179, 189, 234, 235, 236, 237, 238, 239, 245, 246, 247, 248, 249, 257, 258, 259, 267, 268, 269, 278, 279, 289, 345, 346, 347, 348, 349, 356, 357, 358, 359, 367, 368, 369, 378, 379, 389, 456, 457, 458, 459, 467, 468, 469, 478, 489, 567, 568, 569, 578, 579, 589, 678, 679, 689, 789}{
        \addplot[no marks, dash pattern = on 2pt off 1pt] table[x = lambda, y = \u] {data/fistas=0.0sigma=0.1.csv};
      }
      \foreach \u in {138, 256, 479}{
        \addplot[no marks, orange] table[x = lambda, y = \u] {data/fistas=0.0sigma=0.1.csv};
      }
    \end{axis}\end{tikzpicture}
    \caption{$\rho(\bm u, S_{\mathcal{I}_{\b N}(U_3)}^{\mathcal X} f)$ with $|\bm u| = 3$}
  \end{subfigure}
  \caption{FISTA with small bandwidths given in \eqref{eq:smallbandwidths} on noisy data with $s=0$ (orange: active ANOVA terms $U^\star$, dashed: inactive ANOVA terms in the test function).}
  \label{fig:fistas=0sigma=0.1}
\end{figure}
	In the first setting, we choose $\omega(\bm k) = 1$ for all $\bm k\in\mathcal{I}_{\b N}(U)$, which penalizes all frequencies equally, cf.~\cref{sec:lsqr}.
	To compensate for the lack of information we use the small bandwidths \eqref{eq:smallbandwidths} such that we are in an overdetermined setting.
	
	The results of this first experiment without noise can be seen in \cref{fig:fistas=0sigma=0.0}.
	The minimal relative $\L_2$-error is $\eps(\mathcal{X},\mathcal{I}_{\b N}(U_3)) \approx 9.2\cdot 10^{-2}$ which is similar to the observations in \cref{sec:lsqr}.
	We are able to distinguish the global sensitivity indices of the active ANOVA terms very clearly from the ones not occurring in the test function.
	In (b) we observe that the global sensitivity indices $\rho(\bm u,S_{\mathcal{I}_{\b N}(U_3)}^{\mathcal X} f)$
	are larger than zero for all one-dimensional ANOVA terms $f_{\bm u}$ occuring in the test function.
	In (c) and (d) we have the expected sparsifying behaviour from the $\ell_1$-regularization, as many groups are penalized to be exactly zero, for larger $\lambda$ even the ones which occur in the test function. In summary, the group lasso regularization is able to detect the correct active terms. We also recognize the different smoothness of the ANOVA terms as mentioned in experiment (i) from \cref{sec:lsqr}.
	
	\item
	In the next experiment we added $10\%$ Gaussian noise which results in \cref{fig:fistas=0sigma=0.1}.
	The minimal $\L_2$-error has increased to $\eps(\mathcal{X},\mathcal{I}_{\b N}(U_3)) \approx 0.202$.
	We observe the typical behaviour of under- and overfitting in the $\L_2$-error, see (a). This can also be recognized in the global sensitivity indices:
	For small $\lambda$ the ANOVA terms $f_{\bm u}$ not occurring in the test function are not zero in our approximation since it fits the noise.
	For large $\lambda$ on the other hand we also penalize active ANOVA terms $f_{\bm u}$, $\u \in U^\ast$ which can be seen by the orange lines dropping to zero.
	Overall, the active ANOVA terms are distinguishable from the inactive terms and the group lasso is able to correctly detect the active set.
	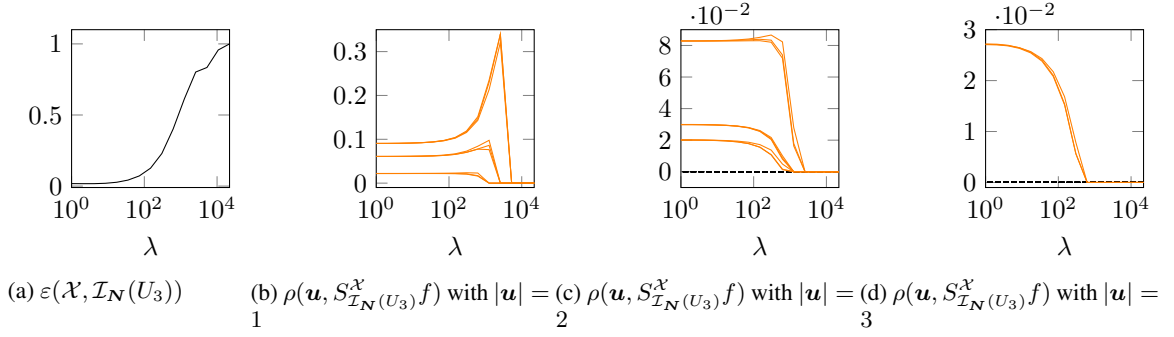
\begin{figure}
  \centering
  \begin{subfigure}[t]{0.24\linewidth}\centering
    \raggedleft
    \begin{tikzpicture}\begin{axis}[
      scale only axis, width = 2.1cm, height = 2.1cm,
      enlarge x limits = 0,
      xmode = log,
      xlabel = $\lambda$,
      ymin = -0.01, ymax = 1.1,
    ]
      \addplot[no marks] table[x = lambda, y = L2error] {data/fistas=1.5sigma=0.0.csv};
    \end{axis}\end{tikzpicture}
    \caption{$\eps(\mathcal{X},\mathcal{I}_{\b N}(U_3))$}
  \end{subfigure}
  \begin{subfigure}[t]{0.24\linewidth}\centering
    \raggedleft
    \begin{tikzpicture}\begin{axis}[
      scale only axis, width = 2.1cm, height = 2.1cm,
      enlarge x limits = 0,
      xmode = log,
      xlabel = $\lambda$,
      ymin = -0.01, ymax = 0.35,
    ]
      \foreach \u in {1, 2, 3, 4, 5, 6, 7, 8, 9}{
        \addplot[no marks, color = orange] table[x = lambda, y = \u] {data/fistas=1.5sigma=0.0.csv};
      }
    \end{axis}\end{tikzpicture}
    \caption{$\rho(\bm u, S_{\mathcal{I}_{\b N}(U_3)}^{\mathcal X} f)$ with $|\bm u| = 1$}
  \end{subfigure}
  \begin{subfigure}[t]{0.24\linewidth}\centering
    \raggedleft
    \begin{tikzpicture}\begin{axis}[
      scale only axis, width = 2.1cm, height = 2.1cm,
      enlarge x limits = 0,
      xmode = log,
      xlabel = $\lambda$,
      ymin = -0.01, ymax = 0.09,
    ]
      \foreach \u in {12, 14, 15, 16, 17, 19, 23, 24, 27, 28, 29, 34, 35, 36, 37, 39, 45, 46, 48, 57, 58, 59, 67, 68, 69, 78, 89}{
        \addplot[no marks, dash pattern = on 2pt off 1pt] table[x = lambda, y = \u] {data/fistas=1.5sigma=0.0.csv};
      }
      \foreach \u in {13, 18, 25, 26, 38, 47, 49, 56, 79}{
        \addplot[no marks, orange] table[x = lambda, y = \u] {data/fistas=1.5sigma=0.0.csv};
      }
    \end{axis}\end{tikzpicture}
    \caption{$\rho(\bm u, S_{\mathcal{I}_{\b N}(U_3)}^{\mathcal X} f)$ with $|\bm u| = 2$}
  \end{subfigure}
  \begin{subfigure}[t]{0.24\linewidth}\centering
    \raggedleft
    \begin{tikzpicture}\begin{axis}[
      scale only axis, width = 2.1cm, height = 2.1cm,
      enlarge x limits = 0,
      xmode = log,
      xlabel = $\lambda$,
      ymin = -0.001, ymax = 0.03,
    ]
      \foreach \u in {123, 124, 125, 126, 127, 128, 129, 134, 135, 136, 137, 139, 145, 146, 147, 148, 149, 156, 157, 158, 159, 167, 168, 169, 178, 179, 189, 234, 235, 236, 237, 238, 239, 245, 246, 247, 248, 249, 257, 258, 259, 267, 268, 269, 278, 279, 289, 345, 346, 347, 348, 349, 356, 357, 358, 359, 367, 368, 369, 378, 379, 389, 456, 457, 458, 459, 467, 468, 469, 478, 489, 567, 568, 569, 578, 579, 589, 678, 679, 689, 789}{
        \addplot[no marks, dash pattern = on 2pt off 1pt] table[x = lambda, y = \u] {data/fistas=1.5sigma=0.0.csv};
      }
      \foreach \u in {138, 256, 479}{
        \addplot[no marks, orange] table[x = lambda, y = \u] {data/fistas=1.5sigma=0.0.csv};
      }
    \end{axis}\end{tikzpicture}
    \caption{$\rho(\bm u, S_{\mathcal{I}_{\b N}(U_3)}^{\mathcal X} f)$ with $|\bm u| = 3$}
  \end{subfigure}
  \caption{FISTA with large bandwidths given in \eqref{eq:largebandwidths} on exact data with $s=1.5$ (orange: active ANOVA terms $U^\star$, dashed: inactive ANOVA terms in the test function).}
  \label{fig:fistas=2.5sigma=0.0}
\end{figure}
	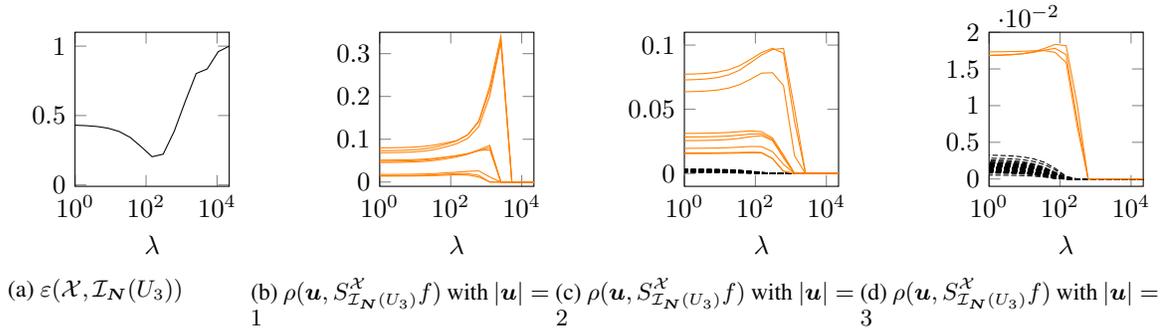
\begin{figure}
  \centering
  \begin{subfigure}[t]{0.24\linewidth}\centering
    \raggedleft
    \begin{tikzpicture}\begin{axis}[
      scale only axis, width = 2.05cm, height = 2.05cm,
      enlarge x limits = 0,
      xmode = log,
      xlabel = $\lambda$,
      ymin = -0.01, ymax = 1.1,
    ]
      \addplot[no marks] table[x = lambda, y = L2error] {data/fistas=1.5sigma=0.1.csv};
    \end{axis}\end{tikzpicture}
    \caption{$\eps(\mathcal{X},\mathcal{I}_{\b N}(U_3))$}
  \end{subfigure}
  \begin{subfigure}[t]{0.24\linewidth}\centering
    \raggedleft
    \begin{tikzpicture}\begin{axis}[
      scale only axis, width = 2.05cm, height = 2.05cm,
      enlarge x limits = 0,
      xmode = log,
      xlabel = $\lambda$,
      ymin = -0.01, ymax = 0.35,
    ]
      \foreach \u in {1, 2, 3, 4, 5, 6, 7, 8, 9}{
        \addplot[no marks, color = orange] table[x = lambda, y = \u] {data/fistas=1.5sigma=0.1.csv};
      }
    \end{axis}\end{tikzpicture}
    \caption{$\rho(\bm u, S_{\mathcal{I}_{\b N}(U_3)}^{\mathcal X} f)$ with $|\bm u| = 1$}
  \end{subfigure}
  \begin{subfigure}[t]{0.24\linewidth}\centering
    \raggedleft
    \begin{tikzpicture}\begin{axis}[
    		yticklabel style={
    			/pgf/number format/fixed,
    			/pgf/number format/precision=5
    		},
    		scaled y ticks=false,
      scale only axis, width = 2.05cm, height = 2.05cm,
      enlarge x limits = 0,
      xmode = log,
      xlabel = $\lambda$,
      ymin = -0.01, ymax = 0.11,
    ]
      \foreach \u in {12, 14, 15, 16, 17, 19, 23, 24, 27, 28, 29, 34, 35, 36, 37, 39, 45, 46, 48, 57, 58, 59, 67, 68, 69, 78, 89}{
        \addplot[no marks, dash pattern = on 2pt off 1pt] table[x = lambda, y = \u] {data/fistas=1.5sigma=0.1.csv};
      }
      \foreach \u in {13, 18, 25, 26, 38, 47, 49, 56, 79}{
        \addplot[no marks, orange] table[x = lambda, y = \u] {data/fistas=1.5sigma=0.1.csv};
      }
    \end{axis}\end{tikzpicture}
    \caption{$\rho(\bm u, S_{\mathcal{I}_{\b N}(U_3)}^{\mathcal X} f)$ with $|\bm u| = 2$}
  \end{subfigure}
  \begin{subfigure}[t]{0.24\linewidth}\centering
    \raggedleft
    \begin{tikzpicture}\begin{axis}[
      scale only axis, width = 2.05cm, height = 2.05cm,
      enlarge x limits = 0,
      xmode = log,
      xlabel = $\lambda$,
      ymin = -0.001, ymax = 0.02,
    ]
      \foreach \u in {123, 124, 125, 126, 127, 128, 129, 134, 135, 136, 137, 139, 145, 146, 147, 148, 149, 156, 157, 158, 159, 167, 168, 169, 178, 179, 189, 234, 235, 236, 237, 238, 239, 245, 246, 247, 248, 249, 257, 258, 259, 267, 268, 269, 278, 279, 289, 345, 346, 347, 348, 349, 356, 357, 358, 359, 367, 368, 369, 378, 379, 389, 456, 457, 458, 459, 467, 468, 469, 478, 489, 567, 568, 569, 578, 579, 589, 678, 679, 689, 789}{
        \addplot[no marks, dash pattern = on 2pt off 1pt] table[x = lambda, y = \u] {data/fistas=1.5sigma=0.1.csv};
      }
      \foreach \u in {138, 256, 479}{
        \addplot[no marks, orange] table[x = lambda, y = \u] {data/fistas=1.5sigma=0.1.csv};
      }
    \end{axis}\end{tikzpicture}
    \caption{$\rho(\bm u, S_{\mathcal{I}_{\b N}(U_3)}^{\mathcal X} f)$ with $|\bm u| = 3$}
  \end{subfigure}
  \caption{FISTA with large bandwidths given in \eqref{eq:largebandwidths} on noisy data with $s=1.5$ (orange: active ANOVA terms $U^\star$, dashed: inactive ANOVA terms in the test function).}
  \label{fig:fistas=2.5sigma=0.1}
\end{figure}
	
	\item
	Now, we are interested in incorporating smoothness information about the function $f$: We choose the large bandwidths \eqref{eq:largebandwidths} which results in $53\,806$ frequencies such that we are underdetermined. With that many degrees of freedom we have to reduce our search space by using the Sobolev weights $\omega_s$ from \eqref{sobolevweights} for $s = 1.5$ in $\bm{\hat W}(\bm u)$. This corresponds to the smoothness of the function since $f \in \sobolev{\omega_{1.5-\eps}}(\T^d)$. 
	
	The results without noise are depicted in \cref{fig:fistas=2.5sigma=0.0}. The behaviour looks very similar to experiment (i) with smaller bandwidths, but the minimal $\L_2$-error is smaller: $\eps(\mathcal{X},\mathcal{I}_{\b N}(U_3)) \approx 1.7\cdot 10^{-2}$.
	Furthermore, under- and overfitting does not occur as the $\L_2$-error curve is strictly increasing. Moreover, the active set of terms $f_{\u}$, $\u \in U^\ast$, is correctly recognized by group lasso. 
	
	\item
	The same experiment with additional $10\%$ Gaussian noise can be seen in \cref{fig:fistas=2.5sigma=0.1}.
	In comparison to experiment (ii) with noise and small bandwidths we achieve a smaller minimal $\L_2$-error $\eps(\mathcal{X},\mathcal{I}_{\b N}(U_3)) \approx 0.203$.
	In addition, the ANOVA terms $f_{\bm u}$, $\u \in U^\ast$ are now better distinguishable as the smoothness weights $\omega_{1.5}$ filter out the non-smooth noise.
\end{enumerate}

\section{Application Data Example}\label{sec:adult}
In this section, we aim to apply our method to the well-known \textit{adult census data set} from the UCI repository \cite{UCI}. The data was extracted from the 1994 US census database. The goal is to use 14 attributes about a person or group of persons to predict whether they have an income of more than $50\,000$ dollars a year. This represents one of the most popular data sets from the UCI repository. 

Since this data is not periodic, we use the generalization hinted at in \cref{rem:extension}. This requires a complete orthonormal system of product structure in a space of non-periodic $\L_2$ functions. We consider the space of square-integrable $\L_2([0,1]^d)$ over the cube $[0,1]^d$ with the half-period cosine basis, i.e., we have functions \begin{equation*}
	\phi_{\k}(\x) = \sqrt{2}^{\abs{\supp \k}} \prod_{j=1}^d \cos(\pi k_j x_j)
\end{equation*} for frequencies $\k \in \N_0^d$ where $\N_0 = \{0,1,\dots\}$. It is a well-known fact that $(\phi_{\k})_{\k \in \N_0^d}$ is a complete orthonormal system in $\L_2([0,1]^d)$, see e.g.\ \cite{KuoMiNoNu19}. Without going into explicit detail, it is evident that the considerations from \cref{sec:anova} and \cref{sec:gft} generalize to this system. However, the low dimensional transforms \eqref{eq:somafm} will be realized trough a non-equispaced fast cosine transform (NFCT), see \cite[Section 7.4]{PlPoStTa18}, instead of the NFFT. 

The data set requires a few simple pre-processing steps in order to be used with our approach. First of all, we remove the attribute \textit{education} since this is already encoded in \textit{education.num} and the attribute \textit{fnlwgt} since it gives a weight to entries which is not normalized across federal states. We also transform categorical attributes to numerical attributes and remove data points with missing values. Moreover, we perform a min-max-normalization such that we have data points in the cube $[0,1]^{12}$. After that pre-processing we have a $45\, 199$ data points which we split into a training node set $\mathcal{X}_{\mathrm{train}} = \{ \x_1, \x_2, \dots, \x_{M_{\mathrm{train}}} \} \subset [0,1]^{12}$ with $M_{\mathrm{train}} = 36\,160$ nodes and a test node set $\mathcal{X}_{\mathrm{test}} = \{ \x_1, \x_2, \dots, \x_{M_{\mathrm{test}}} \} \subset [0,1]^{12}$ with $M_{\mathrm{test}} = 9\,039$ nodes. Therefore, we have $80\%$ of the overall data to obtain the approximation and $20\%$ to verify the performance. Whether a person earns more than $50\,000$ dollars or not is a binary question, i.e., we have a vector $\y_{\mathrm{train}} \in \{0,1\}^{M_{\mathrm{train}}}$ such that for a person $\x_i \in \mathcal{X}_{\mathrm{train}}$ in the training set we know that they earn $50\,000$ dollar or more if $y_i = 1$ and they do not earn as much if $y_i = 0$.

It is our goal to answer the same question for every person from the test set $\mathcal{X}_{\mathrm{test}}$. The general idea is to use an approximate partial sum \begin{equation*}
	S_{\mathcal{I}_{\b N}(U)}^{\mathcal{X}_{\mathrm{train}}} f (\x) \coloneqq \sum_{\k \in \mathcal{I}_{\b N}(U)} \hat{f}_{\k} \,\phi_{\k}(\x) 
\end{equation*} of an unknown function $f \in \L_2([0,1]^d)$ with coefficients $\hat{f}_{\k}$ determined from the methods proposed in \cref{sec:approx}, $U$ a subset of ANOVA terms and a grouped index set $\mathcal{I}_{\b N}(U) \subset \N_0^d$. Now, we use as a model function \begin{equation*}
	\tilde{f}_{\mathcal{I}_{\b N}(U)}^{\mathcal{X}_{\mathrm{train}}}(\x) \coloneqq \begin{cases}
		0 \quad &\colon S_{\mathcal{I}_{\b N}(U)}^{\mathcal{X}_{\mathrm{train}}} f (\x) < 0.5 \\
		1 \quad &\colon S_{\mathcal{I}_{\b N}(U)}^{\mathcal{X}_{\mathrm{train}}} f (\x) \geq 0.5
	\end{cases}
\end{equation*} by deciding with a threshold of $0.5$. The data set provides us with a way of validating our result by providing the correct answer $g(\x) \in \{0,1\}$ for every $\x \in \mathcal{X}_{\mathrm{test}}$ which has not been used in obtaining the model function. As a quality measure (or error), we use the percentage of correctly classified people \begin{equation*}
	p(\mathcal{I}_{\b N}(U),\mathcal{X}_{\mathrm{train}},\mathcal{X}_{\mathrm{test}}) \coloneqq 100 \left( 1 - \frac{1}{M_{\mathrm{test}}}\sum_{\x \in\mathcal{X}_{\mathrm{test}}} \abs{\tilde{f}_{\mathcal{I}_{\b N}(U)}^{\mathcal{X}_{\mathrm{train}}}(\x) - g(\x)} \right) \in [0,100].
\end{equation*}

Our first step is to assume a model with superposition threshold $d_s = 2$, i.e., we use the set of ANOVA terms $U_2 = \{ \u \subset \{1,2,\dots,12\} \colon \au \leq 2 \}$, cf.\ \eqref{eq:pizmosch}. In order to get the grouped index set $\mathcal{I}_{\b N}(U_2)$ we choose parameters \begin{equation}\label{appl:bandwidths}
	N_{\u} = \begin{cases}
		82 \quad&\colon \au = 1 \\ 10 \quad&\colon \au = 2
	\end{cases}
\end{equation} for the related groups of frequencies \begin{equation*}
	\mathcal I_{N_{\u}}^{\bm u, d}
	\coloneqq \left\{
	\bm k\in \N_0^d  \colon \supp \k = \u \text{ and } 1 \leq k_j \leq N_{\u}-1 \text{ for }  j \in \u
	\right\}, \u \in U_2,
\end{equation*} such that $\abs{\mathcal{I}_{\b N}(U_2)} = 6\,319$. Note that all of the following results have been obtained by using 10-fold cross-validation as in \cite{kohavi-nbtree}, i.e., the split of our data into training nodes $\mathcal{X}_{\mathrm{train}}$ and test nodes $\mathcal{X}_{\mathrm{test}}$ has been performed ten times and the results were averaged. This averaging was also performed for the global sensitivity indices. We notice an optimal percentage of correctly classified people $p(\mathcal{I}_{\b N}(U_2),\mathcal{X}_{\mathrm{train}},\mathcal{X}_{\mathrm{test}}) = 86.15\%$ for the $\ell_2$ regularization from \cref{sec:lsqr} and $p(\mathcal{I}_{\b N}(U_2),\mathcal{X}_{\mathrm{train}},\mathcal{X}_{\mathrm{test}}) = 86.08\%$ for the group lasso approach from \cref{sec:fista}

In the case of  $\ell_2$ regularization, we try to improve the error by determining an active set of ANOVA terms. However, the situation is not as clear as for the synthetic example in \cref{sec:approx}. Trough testing, we determined that it delivers the best improvement if we choose $\b\eps = (0.1, 0.1) \in \R^2$ such that terms $f_{\u}$ with global sensitivity indices $\gsi{\u}{S_{\mathcal{I}_{\b N}(U_2)}^{\mathcal{X}_{\mathrm{train}}} f} < 0.1$ are removed for the approximation $S_{\mathcal{I}_{\b N}(U_{\mathcal{X}_{\mathrm{train}},\y}^{(\b\eps)})}^{\mathcal{X}_{\mathrm{train}}}$. We may also increase the bandwidth for the one-dimensional terms such that we have \begin{equation}\label{appl:bandwidths2}
	N_{\u} = \begin{cases}
		300 \quad&\colon \au = 1 \\ 10 \quad&\colon \au = 2.
	\end{cases}
\end{equation} This reduces the number of terms from $\abs{U_2} = 79$ to $|U_{\mathcal{X}_{\mathrm{train}},\y}^{(\b\eps)}| = 35$, cf.\ \eqref{activeset}, and it also increases the number of correctly classified people to a best result of $p(\mathcal{I}_{\b N}(U_{\mathcal{X}_{\mathrm{train}},\y}^{(\b\eps)}),\mathcal{X}_{\mathrm{train}},\mathcal{X}_{\mathrm{test}}) = 86.30\%$. 

In \cref{fig:app} we have visualized our results. From \cref{fig:app:error} it is evident that the group lasso approach does not identify important terms correctly for this data set which yields a lower number of correctly classified people with increasing regularization parameter. 
An explaination is the lack of sparsity in the underlying function. The FISTA algorithm still eliminates many ANOVA terms while sacrificing data fitting quality.
Moreover, for the $\ell_2$ regularization we see that the performance can be improved trough sensitivity analysis and reducing the number of terms. The (averaged) global sensitivity indices of the 12 one-dimensional ANOVA terms are depicted in \cref{fig:app:lsqr:1} for the $\ell_2$ regularization and in \cref{fig:app:fista:1} for the group lasso. Moreover, we have the sensitivity indices for the 66 two-dimensional ANOVA terms are depicted in \cref{fig:app:lsqr:2} and \cref{fig:app:fista:2}. We can see that the group lasso does not identify the same terms as the sensitivity analysis. However, trough the validation on our test set, we observe that the sensitivity analysis yields an active set of terms $U_{\mathcal{X}_{\mathrm{train}},\y}^{(\b\eps)}$ which results in a larger number of correctly classified people.

In \cref{fig:app:lsqr:1} we observe that one term explains around 25\% of the variance of our approximation. This attribute is the \textit{capital loss} and it therefore is the most important attribute according to our model obtained by $\ell_2$ regularization. The same term appears with about 35\% of the variance in \cref{fig:app:fista:1} obtained by group lasso which means that even tough their performance differs, both models identify the same variable as vital to the problem. However, the group lasso approach starts to penalize this term with increasing regularization parameter which may be one reason the performance of this model is dropping.

Note that in order to compare our experiments to other known results, we need to be in a similar setting. We compare to the results from \cite{kohavi-nbtree} and therefore use 10-fold cross-validation to verify our result. Moreover, it is in general possible to perform exploratory data analysis, e.g., remove outliers, which was not done in \cite{kohavi-nbtree} and we also did not perform here. \cref{tab:results:appl} shows the comparison. The ratio of training and test data is with 80\% to 20\% the same. Our approach yields a better classification result than all other approaches from \cite{kohavi-nbtree} which include a support vector machine (SVM), a C4.5 decision-tree induction, a Naive-Bayes approach, and NBtree which is a hybrid between decision-tree and Naive-Bayesian classifiers.

\begin{figure}[tbhp]
  \centering 
  \begin{subfigure}[t]{0.49\linewidth}\centering
    \begin{tikzpicture}\begin{axis}[
      scale only axis, width = 3.5cm, height = 3.5cm,
      enlarge x limits = 0,
      %ymode = log,
      xmode = log,
      xlabel = $\lambda$,
      ytick = {85, 85.5, 86, 86.3},
      ymin = 85, ymax = 86.5,
    ]
      \addplot[no marks] table[x = lambda, y = L2error] {data/lsqr=full.csv};
      \addplot[no marks, color = orange] table[x = lambda, y = L2error] {data/lsqr=active.csv};
	  \addplot[no marks, densely dotted] table[x = lambda, y = L2error] {data/fista.csv};
      \addplot[sharp plot,update limits=false, dashed] 
      coordinates {(1,86.30268835048123) (2980.9579870417283,86.30268835048123)};
    \end{axis}\end{tikzpicture}
    \caption{$p(\mathcal{I}_{\b N}(U_2),\mathcal{X}_{\mathrm{train}},\mathcal{X}_{\mathrm{test}})$ for $\ell_2$ regularization (black) and for group lasso (dotted). $p(\mathcal{I}_{\b N}(U_{\mathcal{X}_{\mathrm{train}},\y}^{(\b\eps)}),\mathcal{X}_{\mathrm{train}},\mathcal{X}_{\mathrm{test}})$ for $\ell_2$ regularization after sensitivity analysis (orange).}\label{fig:app:error}
  \end{subfigure} \\
   \begin{subfigure}[t]{0.44\linewidth}\centering
 	\begin{tikzpicture}\begin{axis}[
 			scale only axis, width = 3.5cm, height = 3.5cm,
 			enlarge x limits = 0,
 			xmode = log,
 			%ymode = log,
 			xlabel = $\lambda$,
 			ymin = -0.005, ymax = 0.25,
 			]
 			\foreach \u in {10,11,12,13}{
 				\addplot[no marks, color = orange] table[x = lambda, y = \u]{data/lsqr=full.csv};
 			}
 			\foreach \u in {2,3,4,5,6,7,8,9}{
 				\addplot[no marks, color = black] table[x = lambda, y = \u]{data/lsqr=full.csv};
 			}
 	\end{axis}\end{tikzpicture}
 	\caption{$\gsi{\u}{S_{\mathcal{I}_{\b N}(U_2)}^{\mathcal{X}_{\mathrm{train}}} f}$, $\au = 1$, for $\ell_2$ regularization and LSQR solver.}\label{fig:app:lsqr:1}
 \end{subfigure}
\hfill%
 \begin{subfigure}[t]{0.44\linewidth}\centering
 	\begin{tikzpicture}\begin{axis}[
 			scale only axis, width = 3.5cm, height = 3.5cm,
			enlarge x limits = 0,
			xmode = log,
			%ymode = log,
			xlabel = $\lambda$,
			]
			\foreach \u in {10,11,12,13}{
				\addplot[no marks, color = orange] table[x = lambda, y = \u]{data/fista.csv};
			}
			\foreach \u in {2,3,4,5,6,7,8,9}{
				\addplot[no marks, color = black] table[x = lambda, y = \u]{data/fista.csv};
			}
 	\end{axis}\end{tikzpicture}
 	\caption{$\gsi{\u}{S_{\mathcal{I}_{\b N}(U_2)}^{\mathcal{X}_{\mathrm{train}}} f}$, $\au = 1$, for group lasso and FISTA solver.}\label{fig:app:fista:1}
 \end{subfigure} \\
\begin{subfigure}[t]{0.44\linewidth}\centering
	\begin{tikzpicture}\begin{axis}[
			scale only axis, width = 3.5cm, height = 3.5cm,
			enlarge x limits = 0,
			xmode = log,
			%ymode = log,
			xlabel = $\lambda$,
			ymin = -0.0025, ymax = 0.07,
			]
			\foreach \u in {14, 15, 16, 17, 22, 23, 24, 27, 33, 34, 36, 40, 41, 42, 43, 44, 49, 50, 51, 55, 56, 57, 58, 62, 63, 68, 74, 77, 79}{
				\addplot[no marks, color = orange] table[x = lambda, y = \u]{data/lsqr=full.csv};
			}
			\foreach \u in {18, 19, 20, 21, 25, 26, 28, 29, 30, 31, 32, 35, 37, 38, 39, 45, 46, 47, 48, 52, 53, 54, 59, 60, 61, 64, 65, 66, 67, 69, 70, 71, 72, 73, 75, 76, 78}{
				\addplot[no marks, color = black] table[x = lambda, y = \u]{data/lsqr=full.csv};
			}
	\end{axis}\end{tikzpicture}
	\caption{$\gsi{\u}{S_{\mathcal{I}_{\b N}(U_2)}^{\mathcal{X}_{\mathrm{train}}} f}$, $\au = 2$, for $\ell_2$ regularization and LSQR solver.}\label{fig:app:lsqr:2}
\end{subfigure}
\hfill%
\begin{subfigure}[t]{0.44\linewidth}\centering
	\begin{tikzpicture}\begin{axis}[
			scale only axis, width = 3.5cm, height = 3.5cm,
			enlarge x limits = 0,
			%ymode = log,
			xmode = log,
			yticklabel style={
				/pgf/number format/fixed,
				/pgf/number format/precision=5
			},
			scaled y ticks=false,
			xlabel = $\lambda$,
			]
			\foreach \u in {14, 15, 16, 17, 22, 23, 24, 27, 33, 34, 36, 40, 41, 42, 43, 44, 49, 50, 51, 55, 56, 57, 58, 62, 63, 68, 74, 77, 79}{
				\addplot[no marks, color = orange] table[x = lambda, y = \u]{data/fista.csv};
			}
			\foreach \u in {18, 19, 20, 21, 25, 26, 28, 29, 30, 31, 32, 35, 37, 38, 39, 45, 46, 47, 48, 52, 53, 54, 59, 60, 61, 64, 65, 66, 67, 69, 70, 71, 72, 73, 75, 76, 78}{
				\addplot[no marks, color = black] table[x = lambda, y = \u]{data/fista.csv};
			}
	\end{axis}\end{tikzpicture}
	\caption{$\gsi{\u}{S_{\mathcal{I}_{\b N}(U_2)}^{\mathcal{X}_{\mathrm{train}}} f}$, $\au = 2$, for group lasso and FISTA solver.}\label{fig:app:fista:2}
\end{subfigure}
  \caption{Numerical experiments with the adult census data set using bandwidths \eqref{appl:bandwidths} in $\mathcal{I}_{\b N}(U_2)$ and bandwidths \eqref{appl:bandwidths2} in $\mathcal{I}_{\b N}(U_{\mathcal{X}_{\mathrm{train}},\y}^{(\b\eps)})$.}
  \label{fig:app}
\end{figure}
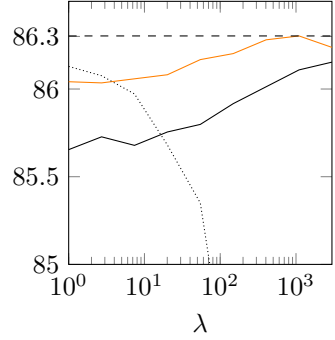
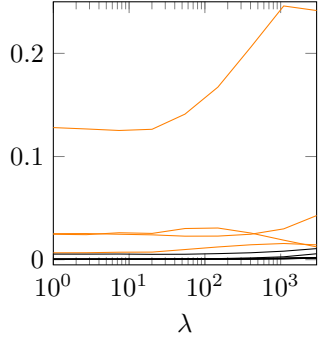
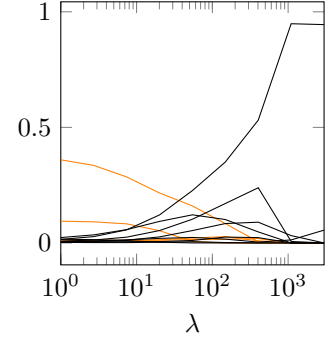
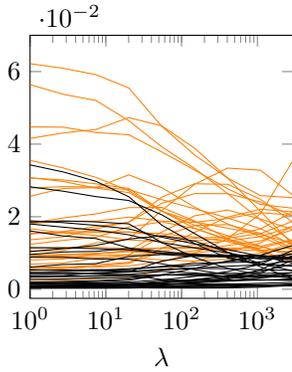
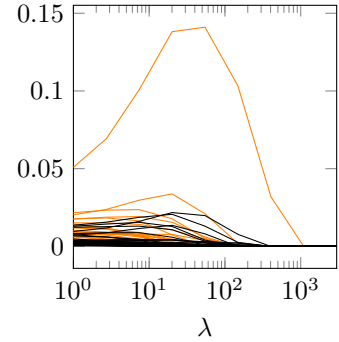

\begin{table}[tbhp]
	\begin{center}
		\begin{tabular}{cc} 
			\toprule
			method & correctly classified persons \\\midrule 
			SVM & 85.03\% \\
			C4.5 & 84.46\% \\
			Naive Bayes & 83.88\% \\
			NBtree & 85.90\% \\
			ANOVAapprox ($\ell_2$) & \textbf{86.30\%} \\
			ANOVAapprox (group lasso) & 86.08\% \\\bottomrule
		\end{tabular}
		\caption{Comparison of the performance of different classification approaches from \cite{kohavi-nbtree} on the adult census data set to methods proposed in this paper.}\label{tab:results:appl}
	\end{center}
\end{table} 

\section{Conclusion}
In the context of Fourier analysis, we introduced grouped frequency index sets $\mathcal{I}_{\b N}(U)$ for a subset of ANOVA terms $U \subset \mathcal{P}(\D)$, see \eqref{groupedset}, consisting of groups $\mathcal I_{N}^{\bm u, d}$, see \eqref{metallica}, along lower-dimensional cubes in the frequency domain. These can then be used to move the computations from the original $d$ dimensions to the dimension of the groups by exploiting the block structure \eqref{bananarama} in the matrix. The resulting transformations are only of dimension up to a superposition threshold $d_s < d$. We presented a fast algorithm --- the grouped Fourier transform --- based on the identity \eqref{eq:somafm}. We were able to carry out every matrix-vector multiplication via the nonequispaced fast Fourier transformation (NFFT) to obtain the typical FFT-like complexity. Moreover, the transformation can be computed simultaneously and independent of the other transformations allowing for parallelization of the algorithm.

Even though we assumed a very special structure, there are proven applications: We have the one-to-one correspondence to the ANOVA decomposition, which we discussed in \cref{sec:anova}. As proven in \cite{PoSc19a}, limiting the variable interactions to a superposition threshold $d_s < d$ works well for functions with a low superposition dimension. Moreover, certain function classes naturally lead to a low superposition dimension for high accuracy. When assuming sparsity-of-effects or the Pareto principle, one may also find that this idea works well for real world problems since those systems are dominated by a few low complexity interactions.

In \cref{sec:approx} we presented two approaches to modify the least-squares problem \eqref{min_prob} occurring in the approximation of the truncated ANOVA decomposition $\mathrm{T}_{d_s} f$, cf.\ \eqref{universe} and the ideas in \cite{PoSc19a}.
\begin{enumerate}[(i)]
	\item
	The first approach follows the concept of minimizing the Tikhonov-functional \eqref{new_min_prob}. This extensions allows for a modification which is able to handle noise and incorporate smoothness information which reduces the \textit{search space} of the minimization to functions of the smoothness class. This enables us to use significantly larger groups $\mathcal I_{N}^{\bm u, d}$ reducing the cut-off error from considering the Fourier partial sum \eqref{eq:rollingstones}, cf.\ \cite[Section 6]{PoSc19a}. 
	\item
	The second approach uses the group lasso idea, where we proposed to use the ANOVA terms for the groups, see \eqref{eq:alligatoah}.
	In contrast to the classical Tikhonov-regularization this promotes sparsity in the ANOVA decomposition directly trough the regularization which eliminates the need for sensitivity analysis and therefore only one minimization needs to be solved. As there is no explicit formula for the solution of this non-smooth minimization problem, we used the FISTA algorithm to tackle it.
\end{enumerate}
Both approaches performed well in the numerical experiments from \cref{sec:approx} which cover the approximation of a synthetic test function in under-, overdertermined, and also noisy settings. It was successful to implement the sensitivity analysis into the regularization via a group lasso approach. Moreover, reducing the search space by considering functions of a certain smoothness also worked well. However, the situation is different for real data, i.e., the adult census data set, in \cref{sec:adult}. Here, we observe that group lasso has trouble identifying important terms and does not perform as well as a manual sensitivity analysis and $\ell_2$ regularization. In summary, we were able to outperform known results from a SVM and other classical machine learning methods on the data set.

The code is available in the Julia packages \texttt{ANOVAapprox} and \texttt{GroupedTransforms} on GitHub, see \path{https://github.com/NFFT/ANOVAapprox} and \path{https://github.com/NFFT/GroupedTransforms}. As shown in \cite{PoSc19a}, it is also possible to work with a huge amount of data, namely multiple million nodes, using these methods. This would of course be a severely overdetermined setting. 

Since we end up with a Fourier representation of the ANOVA decomposition, we can visualize it in the form of an explainable ANOVA network, see \cref{fig:anovanetwork}.
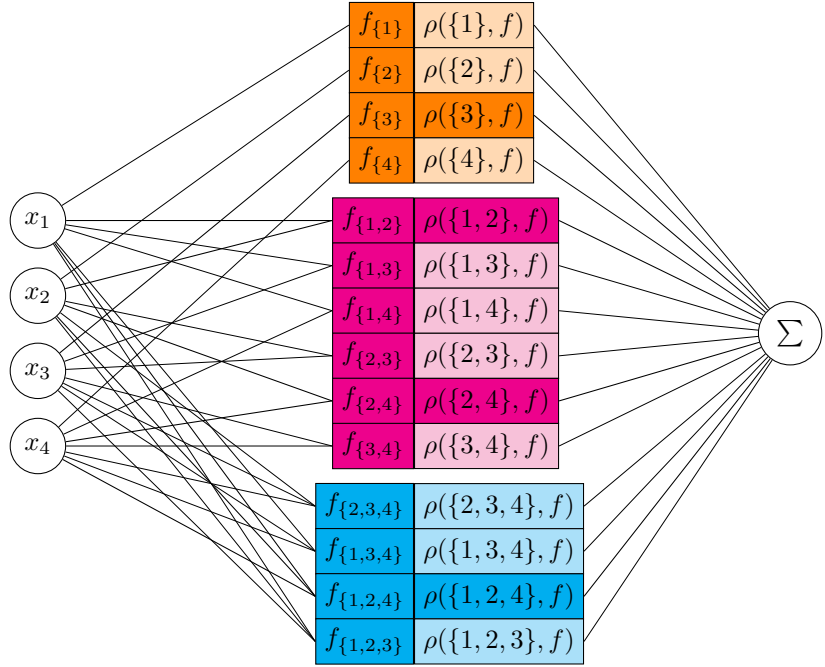
\begin{figure}[tbhp]
  \centering
  \begin{tikzpicture}
    \node (x1) at (-5, 1.5) [circle, draw] {$x_1$};
    \node (x2) at (-5, 0.5) [circle, draw] {$x_2$};
    \node (x3) at (-5, -0.5) [circle, draw] {$x_3$};
    \node (x4) at (-5, -1.5) [circle, draw] {$x_4$};

    \node (f1) at (-0, 4.1) [rectangle, draw, fill = orange, anchor = east, minimum height = 0.6cm] {$f_{\{1\}}$};
    \node (f2) at (-0, 3.5) [rectangle, draw, fill = orange, anchor = east, minimum height = 0.6cm] {$f_{\{2\}}$};
    \node (f3) at (-0, 2.9) [rectangle, draw, fill = orange, anchor = east, minimum height = 0.6cm] {$f_{\{3\}}$};
    \node (f4) at (-0, 2.3) [rectangle, draw, fill = orange, anchor = east, minimum height = 0.6cm] {$f_{\{4\}}$};

    \node (f12) at (-0, 1.5) [rectangle, draw, fill = magenta, anchor = east, minimum height = 0.6cm] {$f_{\{{1,2}\}}$};
    \node (f13) at (-0, 0.9) [rectangle, draw, fill = magenta, anchor = east, minimum height = 0.6cm] {$f_{\{{1,3}\}}$};
    \node (f14) at (-0, 0.3) [rectangle, draw, fill = magenta, anchor = east, minimum height = 0.6cm] {$f_{\{{1,4}\}}$};
    \node (f23) at (-0, -0.3) [rectangle, draw, fill = magenta, anchor = east, minimum height = 0.6cm] {$f_{\{{2,3}\}}$};
    \node (f24) at (-0, -0.9) [rectangle, draw, fill = magenta, anchor = east, minimum height = 0.6cm] {$f_{\{{2,4}\}}$};
    \node (f34) at (-0, -1.5) [rectangle, draw, fill = magenta, anchor = east, minimum height = 0.6cm] {$f_{\{{3,4}\}}$};

    \node (f234) at (-0, -2.3) [rectangle, draw, fill = cyan, anchor = east, minimum height = 0.6cm] {$f_{\{{2,3,4}\}}$};
    \node (f134) at (-0, -2.9) [rectangle, draw, fill = cyan, anchor = east, minimum height = 0.6cm] {$f_{\{{1,3,4}\}}$};
    \node (f124) at (-0, -3.5) [rectangle, draw, fill = cyan, anchor = east, minimum height = 0.6cm] {$f_{\{{1,2,4}\}}$};
    \node (f123) at (-0, -4.1) [rectangle, draw, fill = cyan, anchor = east, minimum height = 0.6cm] {$f_{\{{1,2,3}\}}$};

    \node (g1) at (0, 4.1) [rectangle, draw, fill = orange!30, anchor = west, minimum height = 0.6cm] {$\rho(\{1\}, f)$};
    \node (g2) at (0, 3.5) [rectangle, draw, fill = orange!30, anchor = west, minimum height = 0.6cm] {$\rho(\{2\}, f)$};
    \node (g3) at (0, 2.9) [rectangle, draw, fill = orange, anchor = west, minimum height = 0.6cm] {$\rho(\{3\}, f)$};
    \node (g4) at (0, 2.3) [rectangle, draw, fill = orange!30, anchor = west, minimum height = 0.6cm] {$\rho(\{4\}, f)$};

    \node (g12) at (0, 1.5) [rectangle, draw, fill = magenta, anchor = west, minimum height = 0.6cm] {$\rho(\{{1,2}\}, f)$};
    \node (g13) at (0, 0.9) [rectangle, draw, fill = magenta!30, anchor = west, minimum height = 0.6cm] {$\rho(\{{1,3}\}, f)$};
    \node (g14) at (0, 0.3) [rectangle, draw, fill = magenta!30, anchor = west, minimum height = 0.6cm] {$\rho(\{{1,4}\}, f)$};
    \node (g23) at (0, -0.3) [rectangle, draw, fill = magenta!30, anchor = west, minimum height = 0.6cm] {$\rho(\{{2,3}\}, f)$};
    \node (g24) at (0, -0.9) [rectangle, draw, fill = magenta, anchor = west, minimum height = 0.6cm] {$\rho(\{{2,4}\}, f)$};
    \node (g34) at (0, -1.5) [rectangle, draw, fill = magenta!30, anchor = west, minimum height = 0.6cm] {$\rho(\{{3,4}\}, f)$};

    \node (g234) at (0, -2.3) [rectangle, draw, fill = cyan!30, anchor = west, minimum height = 0.6cm] {$\rho(\{{2,3,4}\}, f)$};
    \node (g134) at (0, -2.9) [rectangle, draw, fill = cyan!30, anchor = west, minimum height = 0.6cm] {$\rho(\{{1,3,4}\}, f)$};
    \node (g124) at (0, -3.5) [rectangle, draw, fill = cyan, anchor = west, minimum height = 0.6cm] {$\rho(\{{1,2,4}\}, f)$};
    \node (g123) at (0, -4.1) [rectangle, draw, fill = cyan!30, anchor = west, minimum height = 0.6cm] {$\rho(\{{1,2,3}\}, f)$};

    \node (sigma) at ( 5, 0) [circle, draw] {$\sum$};

    \draw [ultra thin] (x1) to (f1.west);
    \draw [ultra thin] (x2) to (f2.west);
    \draw (x3) to (f3.west);
    \draw [ultra thin] (x4) to (f4.west);

    \draw (x1) to (f12.west);
    \draw [ultra thin] (x1) to (f13.west);
    \draw [ultra thin] (x1) to (f14.west);
    \draw (x2) to (f12.west);
    \draw [ultra thin] (x2) to (f23.west);
    \draw (x2) to (f24.west);
    \draw [ultra thin] (x3) to (f13.west);
    \draw [ultra thin] (x3) to (f23.west);
    \draw [ultra thin] (x3) to (f34.west);
    \draw [ultra thin] (x4) to (f14.west);
    \draw (x4) to (f24.west);
    \draw [ultra thin] (x4) to (f34.west);

    \draw [ultra thin] (x1) to (f123.west);
    \draw [ultra thin] (x1) to (f134.west);
    \draw (x1) to (f124.west);
    \draw [ultra thin] (x2) to (f123.west);
    \draw [ultra thin] (x2) to (f234.west);
    \draw (x2) to (f124.west);
    \draw [ultra thin] (x3) to (f123.west);
    \draw [ultra thin] (x3) to (f134.west);
    \draw [ultra thin] (x3) to (f234.west);
    \draw (x4) to (f124.west);
    \draw [ultra thin] (x4) to (f134.west);
    \draw [ultra thin] (x4) to (f234.west);

    \draw [ultra thin] (g1.east) to (sigma);
    \draw [ultra thin] (g2.east) to (sigma);
    \draw (g3.east) to (sigma);
    \draw [ultra thin] (g4.east) to (sigma);
    \draw (g12.east) to (sigma);
    \draw [ultra thin] (g13.east) to (sigma);
    \draw [ultra thin] (g14.east) to (sigma);
    \draw [ultra thin] (g23.east) to (sigma);
    \draw (g24.east) to (sigma);
    \draw [ultra thin] (g34.east) to (sigma);
    \draw [ultra thin] (g234.east) to (sigma);
    \draw [ultra thin] (g134.east) to (sigma);
    \draw (g124.east) to (sigma);
    \draw [ultra thin] (g123.east) to (sigma);
  \end{tikzpicture}
  \caption{Explainable ANOVA network of a function $f$ for $d_s=3$. We visualize the different ANOVA terms $f_{\bm u}$ for $|\bm u| =1$ in orange, $|\bm u| =2$ in magenta and for $|\bm u| =3$ in blue. The related global sensitivity indices $\rho(\bm u,f)$ serve as interpretable quantitiy for the ANOVA term.}
  \label{fig:anovanetwork}
\end{figure} This representation immediately shows the couplings of the variables or the ANOVA terms and, with help of the global sensitivity indices \eqref{eq:gsi}, depicts the importance of individual ANOVA terms or, in other words, which terms are active and which are not. 

\section*{Acknowledgments}
Felix Bartel acknowledges funding by the European Social Fund (ESF), Project ID 100367298. Daniel Potts acknowledges funding by the German Research Foundation (Deutsche Forschungsgemeinschaft) - Project ID 416228727 - SFB 1410. Michael Schmischke is supported by the BMBF grant 01$|$S20053A.

\bibliographystyle{siamplain}

\end{document}